\documentclass[11pt,reqno]{amsart}
\usepackage{amssymb,amsmath,genyoungtabtikz,color,cleveref}
\usepackage{fullpage,scalefnt}
\renewcommand{\tilde}{\widetilde}

\renewcommand{\phi}{\varphi}
\newtheorem*{main}{Theorem A}
\newtheorem*{mainB}{Theorem B}

\newtheorem*{mainD}{Conjecture C}
\newtheorem*{mainE}{Theorem D}

\numberwithin{equation}{section}

\crefname{ques}{Question}{Questions}
\crefname{defn}{Definition}{Definitions}
\crefname{thm}{Theorem}{Theorems}
\crefname{prop}{Proposition}{Propositions}
\crefname{lem}{Lemma}{Lemmas}
\crefname{cor}{Corollary}{Corollaries}
\crefname{conj}{Conjecture}{Conjectures}
\crefname{section}{Section}{Sections}
\crefname{subsection}{Subsection}{Subsections}
\crefname{eg}{Example}{Examples}
\crefname{figure}{Figure}{Figures}
\crefname{rem}{Remark}{Remarks}
\crefname{rmk}{Remark}{Remarks}
\crefname{equation}{equation}{equation}

\Crefname{ques}{Question}{Questions}
\Crefname{defn}{Definition}{Definitions}
\Crefname{thm}{Theorem}{Theorems}
\Crefname{prop}{Proposition}{Propositions}
\Crefname{lem}{Lemma}{Lemmas}
\Crefname{cor}{Corollary}{Corollaries}
\Crefname{conj}{Conjecture}{Conjectures}
\Crefname{section}{Section}{Sections}
\Crefname{subsection}{Subsection}{Subsections}
\Crefname{eg}{Example}{Examples}
\Crefname{figure}{Figure}{Figures}
\Crefname{rem}{Remark}{Remarks}
\Crefname{rmk}{Remark}{Remarks}

\newtheorem{theorem}{Theorem}[section]

\newtheorem{lem}[theorem]{Lemma}
\newtheorem{cor}[theorem]{Corollary}
\newtheorem{rem}[theorem]{Remark}
\newtheorem{rems}[theorem]{Remarks}

\newtheorem{prop}[theorem]{Proposition}

\newtheorem{example}[theorem]{Example}
\newtheorem{conj}[theorem]{Conjecture}
\newcommand{\benum}{\begin{enumerate}}
\newcommand{\ennum}{\end{enumerate}}
\newcommand{\N}{\mathbb{N}}
\newcommand{\Z}{\mathbb{Z}}
\newcommand{\C}{\mathbb{C}}

\newcommand{\epps}{\varepsilon}

\newcommand{\gen}[1]{\mbox{$\langle #1 \rangle$}}

\newcommand{\la}{\lambda}
\newcommand{\cla}{[\la]} 
\newcommand{\cmu}{[\mu]}

\newcommand{\al}{\alpha}

\newcommand{\1}{\mbox{\bf 1}}

\newcommand{\down}{{\downarrow}}
\newcommand{\dstyle}{\displaystyle}
\newcommand{\Irr}{{\rm Irr}}
\newcommand{\IBr}{{\rm IBr}}

\DeclareMathOperator{\id}{id}
\DeclareMathOperator{\ord}{ord}
\DeclareMathOperator{\Rem}{Rem}

\newcommand{\sts}{\mathsf{s}}  
\newcommand{\stt}{\mathsf{t}}  

\usepackage{tikz,enumitem,rotating,latexsym,bm,stmaryrd,caption,}
\usetikzlibrary{positioning,intersections,decorations,shadings}
\usetikzlibrary{shapes}
\usetikzlibrary{arrows}
 
\usetikzlibrary{decorations.pathmorphing}
\usetikzlibrary{patterns}

\title[Splitting Kronecker squares]{
Splitting Kronecker squares, \\
2-decomposition numbers, \\
Catalan Combinatorics, \\
and the Saxl conjecture\\[4ex]
}
\author{Christine Bessenrodt}
\author{Chris  Bowman}
\address{
Department of Mathematics,
University of York, Heslington, York, YO10 5DD, UK
}

\begin{document}

\maketitle

\vspace{-0.1cm} 
$$\text{\scriptsize \em  It is with much sadness that we mark the passing of Christine Bessenrodt,   a great friend and mathematician.} $$

\begin{abstract}
This paper 
concerns the 
   symmetric and anti-symmetric Kronecker products of  
characters of the symmetric groups.  
We   provide new closed formulas for decomposing these products, 
 unexpected  connections with 2-modular decomposition numbers, Catalan combinatorics, 
and a refinement  of  the famous  Saxl conjecture.      
\end{abstract}
  
\maketitle
\thispagestyle{empty}

    \section{Introduction}\label{sec:intro}
 The Kronecker problem asks for  an understanding of the
  tensor products of   characters of  symmetric groups. 
 Given $ \la \vdash n$ a partition of $n$, we let $[\la]$ denote the corresponding simple
   $\mathbb C S_n$-character.
 The Kronecker coefficients $g(\lambda,\mu,\nu) $ encode  the multiplicities 
\begin{align}\label{jhhj2}\textstyle 
[\lambda][\mu]= \sum_{\nu} g(\lambda,\mu,\nu) [\nu]  .
\end{align}
  The {Kronecker coefficients}   have been described as 
    `perhaps the most challenging, deep and mysterious objects in algebraic combinatorics' \cite{PP}.  
%
    Richard Stanley identified   the calculation of   Kronecker coefficients as  one of the definitive 
     open problems in algebraic combinatorics \cite[Problems 9]{Sta00}. 
   The positivity of Kronecker coefficients is equivalent to the existence of certain quantum systems  \cite{ky,MR2197548,MR2276458}  and  they have been used to understand entanglement entropy \cite{MR3748296}.

The  Kronecker squares  decompose  as   sums of   symmetric and anti-symmetric parts;   we hence define  the symmetric and anti-symmetric Kronecker coefficients 
\begin{align}\label{jhhj}
\textstyle 
[\la] [\la] = S^2[\la]+A^2[\la]\qquad
S^2[\la]=\sum_\nu {\rm sg}(\la,\nu)\color{black} [\nu] \color{black}
\qquad
A^2[\la]=\sum_\nu {\rm ag}(\la,\nu)\color{black} [\nu] \color{black}.
\end{align}
The question of the irreducibility of symmetric and  anti-symmetric tensor products is a central problem in group theory, where it is key to the Aschbacher--Scott maximal subgroup programme \cite{MM,MR1287249,MR1888172}.  
 For $\la$ a rectangular partition, the  coefficients  $ {\rm sg}(\la,\nu)$   played a starring  role in the demise of several famous conjectures Geometric Complexity Theory \cite{MR3695867,MR3868002}.   
 Despite  their central importance, almost nothing is known about the symmetric and anti-symmetric Kronecker coefficients.  In particular, much less is known about the  coefficients in (\ref{jhhj}) than their classical counterparts in (\ref{jhhj2}). 

In this paper, we take the first significant steps towards understanding symmetric and anti-symmetric Kronecker products.  
We   provide new closed formulas for decomposing these products, 
 unexpected  connections with 2-modular decomposition numbers, Catalan combinatorics, 
and a refinement  of  the famous  Saxl conjecture.      

\subsection{Generalising milestones from the classical theory}
The bulk of the paper  is dedicated to advancing our understanding of  
symmetric and anti-symmetric Kronecker coefficients by analogy with well-known milestones   in the classical theory.  
 The  milestones we generalise include: the classification of homogeneous and irreducible products \cite{BK}; the classification of multiplicity-free Kronecker products 
\cite{BeBo}; partial and complete results for special classes of partitions (such as hooks \cite{Blasiak,Liu}, 2-line partitions \cite{BaOr,Rosas}, partitions of small depth  \cite{Saxl, V, Z}, and rectangles \cite{Man10,Man11}); and most recently, Saxl's Kronecker positivity conjecture.

 We provide the analogue of the Bessenrodt--Kleshchev classification of multiplicity-free products for the symmetric and anti-symmetric Kronecker squares.   Unlike in the case of   classical Kronecker products, we do find that there exist  
non-linear  homogeneous  anti-symmetric Kronecker products. 
 
 \begin{main} 
 Any symmetric product $S^2([\la])$ for $\la \vdash n$ is (reducible and) inhomogeneous unless $\la$ is a linear partition.   
Any anti-symmetric product $A^2([\la])$ for $\la \vdash n$ is (reducible and) inhomogeneous unless $\la=(n), (n-1,1)$, $(2^2)$, or  $(3^2)$  (up to conjugation).  
 \end{main}

The corresponding classification for  plethysm products  was obtained in \cite{bowP}. 
 The classification of multiplicity-free symmetric and anti-symmetric Kronecker products is the subject of  \Cref{prop:mfsym-and-alt,prop:mfsym-and-alt2};   
 when both parts are multiplicity-free, a complete answer is given.

The  symmetric and anti-symmetric Kronecker squares of   hook characters  were  recently determined  in 
\cite{MW}; we provide   new results on the hook constituents in arbitrary 
symmetric and anti-symmetric Kronecker products 
in \cref{hook-constituents}.  We provide the complete decomposition  of 
$S^2([k,k])$ and $A^2([k,k])$
 in \Cref{thmkksplit} (see also Theorem B) as well as 
$S^2([k+1,k])$, $A^2([k+1,k])$
and 
$S^2([k+1,k-1])$, $A^2([k+1,k-1])$
 in \Cref{thmk+1ksplit,thmk+1k-1}. 
For arbitrary $\la$, we determine 
  the constituents of small depth in $S^2([\la])$ and $A^2([\la])$   in \Cref{thm:smalldepthconst} 
   and we obtain stronger results for $\la$ a rectangular partition in \cref{rectangleresult}
   (these results are also an essential part of  our proof of the 
  classification of multiplicity-free products for the symmetric and anti-symmetric Kronecker squares).
For   $\la$ of small depth we obtain  the complete decompositions in  \cref{somesmallprods}.   
In Section~\ref{sec:sign-and-beyond} we look at the opposite end of the spectrum 
and locate the sign character (giving an alternative proof of a recent
result of \cite{GIP}) and its neighbour within $A^2[\la]$.
 
 \smallskip
\subsection{A new Catalan identity} The (symmetric) Kronecker coefficients 
 have been intensely studied in recent years,   
   motivated by  applications 
across 
 invariant theory, geometric complexity theory, 
and quantum information theory.  
We  provide  a new application 
to  algebraic  combinatorics and answer a question  posed by Laurent Manivel in 2010, \cite{Man10}. 

\begin{mainB}
For $k\in \N$ and $n=2k$, we have
\begin{align}\label{akabeldsjds}
S^2([k,k]) =\sum_{\alpha   \in E_4(n)} [\alpha  ] \;,\quad
A^2([k,k]) =\sum_{\alpha   \in O_4(n)}[\alpha  ]\: 
\end{align}
where $E_4(n)$ and $O_4(n)$ denote the sets of partitions 
of the form $\la=(\la_1,\la_2,\la_3,\la_4)$ with all 
$\la_i \in 2\N$   (or $\la_i \in 2\N+1$, respectively).  
This result has the following combinatorial shadow. 
 We  define $$s(\al)=\begin{cases}
1 &\text{ if }\alpha\in E_4(n)\\
-1 &\text{ if }\alpha\in O_4(n)\\
0 &\text{ otherwise}. 
\end{cases}$$  
Letting $C_k=\frac{1}{k+1}{2k \choose k}$ denote the $k$-th Catalan number, we have that 
$$C_k = \sum_{\al \vdash 2k} s(\al) f(\al)   
$$
 where   
  $f(\al)  =[\al](\id )$ is the number of standard Young tableaux of shape~$\al$.

\end{mainB}

 The result for Kronecker squares $[k,k]^{\otimes 2}$ 
received a great detail of attention a decade ago \cite{GWXZ,BWZ10,Man10}  in part due 
\color{black}to  \color{black}
motivation from   theoretical quantum computation.
It was in Manivel's paper concerning these tensor squares that he 
posed the question as to the decomposition of $[k,k]^{\otimes 2}$  into symmetric and anti-symmetric parts, which Theorem B resolves.  
The first few examples of Catalan numbers can be calculated using Theorem B as follows,
\begin{align*}
C_1&=1=1=f({\Yvcentermath1\Yboxdim{4pt}\gyoung(;;)})
\\
C_2&=2=1+2-1=f({\Yvcentermath1\Yboxdim{4pt}\gyoung(;;;;)})+f({\Yvcentermath1\Yboxdim{4pt}\gyoung(;;,;;)})-f({\Yvcentermath1\Yboxdim{4pt}\gyoung(;,;,;,;)})
\\
C_3&=5=1+9+5-10=f({\Yvcentermath1\Yboxdim{4pt}\gyoung(;;;;;;)})+f({\Yvcentermath1\Yboxdim{4pt}\gyoung(;;;;,;;)})
+f({\Yvcentermath1\Yboxdim{4pt}\gyoung(;;,;;,;;)})
-f({\Yvcentermath1\Yboxdim{4pt}\gyoung(;;;,;,;,;)}).
\end{align*}

\smallskip
 \subsection{Saxl's conjecture}
A few years ago, fresh impetus for the Kronecker problem came
from a conjecture of Saxl which states that for a triangular number $n=\tfrac{1}{2}k(k+1)$ and  
 $\rho_k=(k,\ldots,2,1)\vdash n$  the staircase partition, the tensor square $[\rho_k]^2$ contains
  every irreducible $\mathbb C S_n$-character with positivity multiplicity.  
Whilst Saxl's conjecture is still unverified, 
many constituents of $[\rho_k]^2$ have been found and 
 the conjecture has inspired a lot of recent work, some using
connections to other groups,
or having applications to Geometric Complexity Theory and
Quantum Information Theory.  
We suggest the following strengthening of Saxl's conjecture:

\begin{mainD}
The symmetric part  $S^2([\rho_k])$ of the square $[\rho_k]^2$
contains all irreducible characters $[\la]$ of $S_n$ as constituents,
except for the character $[1^n]$ when $k\equiv 2 \mod 4$.

\end{mainD}

We also formulate (anti-)symmetric generalisations of the Heide--Saxl--Tiep--Zalesskii conjecture \cite{HSTZ}, see \cref{sec:Saxl-refined} for more details. 

\smallskip
\subsection{Kronecker splitting and  2-modular representation theory}  
In Section~\ref{sec:rel-to-dec}
we provide a surprising new link  between 
the splitting of Kronecker squares of complex characters
 and the calculation of 2-modular decomposition numbers.  
This serves both ways: as a further motivation for studying the splitting of Kronecker squares
(we will see that the square splitting provides new linear
relations between 2-decomposition numbers) 
 but also for obtaining results on Kronecker squares by using 2-projective characters.
For example, using the fact that triangular partitions label simple projective modules in characteristic 2,  we obtain the following:

\begin{mainE}
Let $k\in \N$, $n=k(k+1)/2$ and $\rho_k=(k,k-1,\ldots,1) $ the staircase partition.  Given $\la \neq \rho_k$, we have that 
$$\gen{S^2(\cla),[\rho_k]} = \gen{A^2(\cla),[\rho_k]}$$
and, in particular, $g(\la,\la,\rho_k)$ is even.
\end{mainE}

Our 2-modular results are also key to the proof of Theorem A, above.  
Further connections between Kronecker splittings and 
 2-decomposition numbers are discussed
in some detail in \cref{sec:rel-to-dec} 
and
\color{black} are  \color{black}
 used to calculate information about hook constituents of   symmetric and anti-symmetric Kronecker products in \cref{hook-constituents}.  
These modular results are also used later in the proof of Theorem A.  
We hope this  should add further interest in the problem
of determining the splitting of Kronecker squares.

\smallskip
\subsection{Existing   literature on symmetric Kronecker coefficients} 
The  symmetric Kronecker coefficients are of fundamental importance in 
 Geometric Complexity Theory (see for example \cite{BCI,BCJ,GIP,BLMW} and references therein).  
Despite this interest,  almost nothing is known about values of symmetric Kronecker coefficients: 
the ${\rm sg}((a,1^b),\la)$ for $\la \vdash a+b$ were calculated in \cite{MW};  
 some examples of zero values were calculated in \cite{Ressayre}; the irreducible (anti-)symmetric squares for {\em alternating} groups were classified in \cite{MM} (over fields of arbitrary characteristic).  
 In this paper we provides new tools for the calculation  of symmetric (and anti-symmetric) Kronecker coefficients and suggest  further avenues of research.

      \section{Preliminaries}\label{sec:prelim}

 In this section we introduce some notions,   fix some notation,
and we also recall some background.
 
 \subsection{Partition combinatorics }
 We define a partition, $\la \vdash n$, to be a finite, weakly decreasing sequence of non-negative integers $(\la_1,\la_2,\dots)$  whose sum $|\la|=\la_1+\la_2+\cdots$ equals $n$.  
 We denote by $P(n)$ the set of all partitions of $n$.  
 For a partition $\lambda \in P(n)$, we write $|\la|=n$ for its size and $\ell (\lambda)$ for its length,
i.e., the number of positive parts of $\lambda$.
We define
$P_\ell(n)$ to be the set of all partitions $\la=(\la_1,\la_2,\dots,\la_\ell)$ with  at most  $\ell$   positive parts.  The Young diagram of $\la$ is given as
$$
Y(\la) = \{(i,j) \mid i\in \{1,\ldots,\ell(\la)\}, j\in \{1,\ldots,\la_i\}\};
$$
we think of it as a diagram depicted in matrix notation, with
a box at each $(i,j)\in Y(\la)$, and row~1 being the top row.
The partition $\la$ and its diagram are occasionally identified,
e.g., when we talk about an intersection $\la \cap \mu$ of two
partitions $\la,\mu$.  Given $\nu$ and $\la$ two partitions, we write $\la \subseteq \nu$ if $\la_i\leq \nu_i$ for all $i\geq 1$.  Given $\la\in P(m)$ and $\nu \in  P(n)$ such that  $\la \subseteq \nu$, we  define the resulting skew-partition $\nu \setminus \la$     (or $Y(\nu\setminus\la)$) of $n-m$  to be  the set difference   $Y(\nu)\setminus Y(\la)$.  
We say that $\la\vdash n$ is  a linear partition if its Young diagram is a line, that is $\la\in \{(n),(1^n)\}$.

When $\la=(\la_1,\la_2,\ldots ) \in P(n)$, its depth is
defined to be $d(\la)=n-\la_1$.  
For two partitions $\la=(\la_1,\la_2,\ldots)$, $\mu=(\mu_1,\mu_2,\ldots)$, we define
their sum componentwise, i.e., $\la+\mu=(\la_1+\mu_1, \la_2+\mu_2, \ldots)$ (where we extend
the partitions by trailing zeros, if necessary).
When the smallest part of $\la$ is greater than \color{black} or  \color{black} equal to $  \mu_1$,  
we can concatenate the parts of $\la$ and $\mu$ and we 
denote the resulting partition by $\la \cup \mu$.

An important notion in the theory is that of a hook in a diagram
(see \cite[Section 2.3.17]{JK} for more on this).
The hook $H_{ij}$ to $(i,j)\in Y(\la)$ is the set of boxes
$$
H_{ij}(\la) = \{(i,s) \mid s\in \{j,\ldots,\la_i\}\}
\cup \{(r,j) \mid r\in \{i,\ldots,\la^t_j\}\};
$$
its length is $h_{ij}=|H_{ij}(\la)|$; a hook of length $k$ is also called
a $k$-hook.
The diagonal   (or Durfee) length  of $\la$, denoted by $dl(\la)$,
is the number of \color{black} non-zero  principal hooks $h_i:=H_{ii}(\la)$ for $i\geq 1$.
 We let $H(\la)$ denote the partition formed from the principal hook lengths of $\la$, that is 
 $H(\la)=(h_1,h_2,\dots, h_{dl(\la)})$.  For example 
 $H(4,3,3)=(6,3,1)$.
   \color{black}

 Fixing $k$, and successively removing $k$-hooks from $\la$
as long as possible,
we reach the (uniquely determined) $k$-core $\la_{(k)}$ of $\la$;
the number of $k$-hooks that we have removed from $\la$ to obtain its $k$-core
is called the $k$-weight of~$\la$.
For $k=2$, the 2-cores are just
the partitions of staircase form
$\rho_m=(m,m-1,m-2,\ldots,2,1)$, for some $m\in \N_0$
(for $m=0$ considering $\rho_0$ as the empty partition).

Given $\nu\setminus \lambda $ a (skew) partition of $n$, we define a $(\nu\setminus\la)$-tableau of weight $\mu$ to be a map 
${\sf T}: Y(\nu\setminus\la)  \to \{1,\dots,n\}$ such that $\mu_i=|\{ x \in Y(\nu\setminus \lambda) : {\sf T}(x)=i  \}|$ for $i\geq 1$.     
We depict this by placing each integer within the corresponding box in the Young diagram, for example 
$$
\gyoung(;;1;2,3;4)
\qquad
\gyoung(;;2;1,3;4)
\qquad
\gyoung(;;1;2,4;3)
\qquad
\gyoung(;;2;1,4;3)
\qquad
\gyoung(;;1;1,2;3)
\qquad
\gyoung(;;1;1,3;2)
$$
are all examples of $((\color{black} 3,2 \color{black})\setminus (1))$-tableaux; the first 4 tableaux are of weight $(1^4)$ and the final two tableaux  are of weight $(2,1^2)$.  
We define an equivalence relation on $\nu\setminus \la$-tableaux of weight $\mu$ 
by 
$\sf  S \sim T$ 
if $\sf  S$ and $\sf  T$ differ only by permuting the entries within their   rows.  
For example,  the first 4 tableaux above (respectively the final 2 tableaux above) belong to the same equivalence class $\sim$.  
We define a row-standard tableau to be an equivalence class  of $\sim$ and we 
choose as a  
$\sim$-class representative to be  that in which the entries along the rows are weakly  increasing.  
We say that a row-standard tableau is semistandard if the entries along its columns are strictly increasing.  
We denote the sets of all row-standard and semistandard $\nu\setminus \la$-tableaux of weight $\mu$ 
by ${\sf RStd}(\nu\setminus \la,\mu)$ and  ${\sf SStd}(\nu\setminus \la,\mu)$, respectively.   
 We define the reverse reading word  of a tableau,  $\sf T$,
 to be the sequence of integers obtained by recording the entries of 
the first row of  $\sf T$ backwards, followed by the second row, and continuing in this fashion.  
For example the reverse reading words of the above tableaux are
$$
2143 \quad 1243
\quad
2134
\quad
1234 
\quad
1132
\quad
1123
$$
respectively.  We define the set of Littlewood--Richardson tableaux, ${\sf LR} (\nu\setminus \la,\mu)
\subseteq {\sf SStd} (\nu\setminus \la,\mu)$ whose reverse reading word is a lattice \color{black}word
\color{black} (i.e. every \color{black} left subfactor \color{black} has more $j$s than $(j+1)$s for $j\geq 1$), \color{black}   we refer to \cite[Chapter 16]{J} for a less terse definition of these tableaux. 
\color{black}  \color{black}
 
 \subsection{Representations of symmetric groups}
We write $S_n$ for the symmetric group on $n$ letters.
For background on the representation theory of the symmetric groups, the reader is referred to \cite{J,JK}.
 For $\lambda \in P(n)$,  we let $S^\la_\Z$ and $M^\la_\Z$ denote  the Specht  and Young permutation modules for $S_n$
associated to $\la$, defined over~$\Z$.    
Given $F$ a field, we set  $S^\la_F=S^\la_\Z\otimes _\Z F $
 and $M^\la_F=M^\la_\Z\otimes _\Z F $.  
The permutation module $M^\la_\Z$ can be constructed as  
\color{black} having basis  indexed by the   \color{black}
 set  ${\sf RStd}(\la)$ under  the symmetric group action by place permutation (modulo the equivalence class $\sim$).  
By classical rules of Young and Littlewood--Richardson, we  have that
$$M^{{\color{black}  \nu \color{black}}}_{\mathbb C} \cong \bigoplus_{\nu \in P(n)} |{\sf SStd}(  \la,\nu)| S^{{\color{black} \la \color{black}}}_{\mathbb C} \qquad 
{\rm ind}_{S_m\times S_n}^{S_{m+n}}(S^\la_{\mathbb C} \boxtimes S^\mu_{\mathbb C}) \cong \!\!\!\bigoplus_{\nu \in P(m+n)} |{\sf LR}(\nu\setminus \la,\mu)| S^\nu_{\mathbb C}.$$
%

For $S^\la_\C$, we write $\cla$ for the corresponding
irreducible complex character of $S_n$.
Then
$\{[\la] \mid \la \in P(n)\}$
is the set of all irreducible complex characters of $S_n$.
When $\la=(\la_1,\ldots,\la_m)$,
we omit the parentheses and write $[\la_1, \ldots, \la_m]$;
in particular, $[n]$ is the trivial character of $S_n$.

When we evaluate $[\la]$ on an element of $S_n$ of cycle type $\mu \in P(n)$,
we simply write $[\la](\mu)$ for the corresponding value.
 For $\la\in P(n)$,  we write $f(\la)  =[\la](\id)$ for the degree of $\cla$.


Of central interest in the representation theory
of the symmetric groups
are the Kronecker coefficients $g(\la,\mu,\nu)$
appearing as expansion coefficients in the
Kronecker products
$$
[\lambda][\mu]= \sum_{\nu} g(\lambda,\mu,\nu) [\nu] ,
$$
\color{black}  where 
$[\lambda][\mu]$ is the character defined by 
$[\lambda][\mu](g)=([\lambda](g))  ([\mu](g))$ for $g \in S_n$.  
\color{black}  \color{black}
Our main topic here is to split Kronecker squares into their symmetric and
alternating parts.

\subsection{Symmetric and anti-symmetric squares of representations}

We now recall some well known notions and
results on the symmetric and alternating parts
of a tensor square (see \cite{Huppert, I-book}).
 For any finite group $G$, we denote by $\Irr(G)$ the set of its irreducible (complex) characters, and $\gen{-,-}$ will denote
the usual scalar product on the $\C$-vector space of class functions on~$G$.

Now let $F$ be a field and $V$ a (finite-dimensional) $FG$-module; we consider its tensor square
$V\otimes_F V$, on which $G$ acts diagonally.
Let $\tau: V\otimes V \to  V\otimes V$ be the $FG$-homomorphism
defined on elementary tensors by
$$\tau(v_1 \otimes v_2) = v_2 \otimes v_1.$$
Let $S^2(V)$ and $A^2(V)$, respectively, be the
eigenspaces to $1$ and $-1$ for $\tau$;
these are the symmetric
and alternating part of the tensor square, respectively;
$A^2(V)$ is also called the  antisymmetric or exterior part of the square.
When $F=\C$, 
one easily computes
the characters of $S^2(V)$ and $A^2(V)$ from the character to~$V$; we recall the formula here:

\begin{lem}\label{lem:SA-values}
Let $V$ be a $\C G$-module with character $\chi$. 
Then the character $\chi_S$ of $S^2(V)$ is given by
\begin{align}
\label{nonzero1}
 \chi_S(g)=\tfrac 12 (\chi(g)^2+\chi(g^2)), \; \text{for all } g \in G.
 \end{align}
The character $\chi_A$ of $A^2(V)$ is given by
\begin{align}
\label{nonzero2}
 \chi_A(g)=\tfrac 12 (\chi(g)^2-\chi(g^2)), \; \text{for all } g \in G.
 \end{align}
\end{lem}

In particular, for a character $\chi$,
the class function $\chi^{(2)}$ defined by
\[
\chi^{(2)}(g) = \chi(g^2) \quad \text{for all } g \in G
\]
is a difference of two characters, namely
\begin{align}\label{ppppp2}
 \chi^{(2)} = \chi_S - \chi_A .
 \end{align}
We denote by  $\1_G$ the trivial character of~$G$.
Then
\[
\gen{\chi^{(2)},\1_G }= \frac{1}{|G|} \sum_{g\in G} \chi(g^2) =: \nu_2(\chi)
\]
is the Frobenius--Schur indicator for $\chi$.
It is well known that
$\nu_2(\chi)\in \{-1,0,1\}$, and that this value is nonzero if and only if $\chi$ is a real-valued character, 
and it is ~1 exactly if $\chi$ is the character of a real representation of~$G$ (see \cite[13.1]{Huppert} or \cite[Chap.~4]{I-book}).
Thus, when $\chi$ is the character of a real representation of $G$,
we have
\[
1=  \nu_2(\chi) = \gen{\chi^{(2)},\1_G } = \gen{\chi_S - \chi_A, \1_G},
\]
and  $1=\gen{\chi,\chi} = \gen{\chi^2,\1_G}$,
hence $\gen{\chi_S,\1_G}=1$ and $\gen{\chi_A,\1_G}=0$.

Since all irreducible complex characters of $S_n$ are  
characters of rational representations, we have
these properties for all $\chi \in \Irr(S_n)$, i.e.,
the trivial character $[n]$ is a constituent of $S^2(\cla)$,
but not of $A^2(\cla)$, for all $\la \in P(n)$.

The main aim of our investigations is to contribute several results
on the characters $S^2(\cla)$ and $A^2(\cla)$, i.e.,
to provide information on the coefficients
$sg(\la,\mu), ag(\la,\mu)\in \N_0$ defined for $\la\in P(n)$ by
\[
S^2(\cla) = \sum_{\mu \in P(n)} sg(\la,\mu) \cmu
\quad  \text{and } \quad
A^2(\cla) = \sum_{\mu \in P(n)} ag(\la,\mu) \cmu .
\]

\begin{rem}
We note that $A^2(\la)=0$ if and only if $\la \in \{(n),(1^n)\}$.  
To see this, note that $\chi_\la({\rm id}) > 1$ for all $\la  \not \in \{(n),(1^n)\}$ and so the character is non-zero by equation (\ref{nonzero2}).  
We note that $S^2(\la)$ is never zero (for instance, we have already seen that it always contains the trivial character).  
\end{rem}

In later sections, the following facts will be useful
(they only require short computations with the explicit values from Lemma~\ref{lem:SA-values}):

\begin{lem}\label{lem:SA-res}
\label{lem:SA-sum}
Let $G$ be a finite group and $U$ a subgroup of $G$.
Let $\chi$ be a character of $G$, and suppose
the restriction to $U$ decomposes as
$\chi{\downarrow}_U = \chi_1 + \chi_2$,
with characters $\chi_1,\chi_2$ of $U$.
Writing $X$ for either $S$ or $A$, we have
\[
X^2(\chi){\downarrow}_U \, =
X^2(\chi_1) + X^2(\chi_2) + \chi_1  \chi_2.
\]
\end{lem}


\begin{lem}\label{lem:SA-prodgroups}
Let $G, H$ be finite groups.
Let $\chi$ be a character of $G$,
and $\psi$ a linear character of $H$.
Writing $X$ for either $S$ or $A$, we then have
\[
X^2(\chi  \psi) =
X^2(\chi)   \psi^2.
\]
\end{lem}

 While there is a strong monotonicity property for Kronecker coefficients,
for the symmetric and alternating coefficients less is known.
We state here the semigroup property   recently proven by Ressayre~\cite{Ressayre}.

\begin{prop}\cite[Proposition 2]{Ressayre}\label{prop:Ressayre}
Given $m\in \N$, we define
 $$L_m=\{ (\la,\mu)\mid   \la \in P_m(n), \mu\in P_{ m^2}(n) \text{ and } sg(\la,\mu)\ne 0\}.$$ 
Then, as a subset of $\Z^{m+m^2}$,
$L_m$ is a finitely generated semigroup.   
In particular, if $\alpha,\beta \vdash n_1$ and $\lambda,\mu \vdash n_2$ 
are such that $sg(\alpha,\beta )>0$ and $sg(\lambda,\mu)>0$ then we have have that 
 $$sg(\alpha+\lambda,\beta+\mu)>0.$$

\end{prop}

  \section{The class functions $\chi^{(m)}$ and modular decomposition numbers}
\label{sec:rel-to-dec}

In this section we want to explain how information on the splitting of squares $[\la]^2$ gives information on 2-decomposition numbers,
i.e., computing character data leads to information on the
 compositions series  of the (non-simple) Specht modules $S^\la_F$ for $F$ a field of characteristic 2.

But first we consider a more general situation.
Generalising the class functions $\chi^{(2)}$ defined
in the previous section, one
defines for a character $\chi$ of a finite group~$G$ and $m\in \N$ the class function $\chi^{(m)}$  by $\chi^{(m)}(g)=\chi(g^m)$, for all $g\in G$.
%
This class function is also known to
be a difference of two characters \cite{I-book}.


We start with an easy but useful observation.
 Let  $C=x^G$ be a conjugacy class of the finite group $G$.
Then we define a class function $\vartheta _C$ by
\[
\vartheta _C = \sum_{\psi \in \Irr(G)} \psi(x^{-1}) \psi .
\]
We note that by column orthogonality, we have
\[
\vartheta _C(g)=
\left\{
\begin{array}{cl}
|C_G(x)| & \text{if } g \in C=x^G\\
0 & \text{otherwise.}
\end{array}
\right.
\]

\begin{lem}\label{lem:theta_C} 
Let $G$ be a finite group, $\chi$ a character of $G$, and $m\in \N$.
For any conjugacy class $C=x^G$ of $G$ we have
\[
\gen{\chi^{(m)},\vartheta _C} = \chi(x^m).
\]
\end{lem}

\begin{proof}
We have
\[
\gen{\chi^{(m)},\vartheta _C}
=\frac{1}{|G|} \sum_{g\in G} \chi(g^m)\vartheta _C(g)
= \frac{1}{|G|} \sum_{g\in x^G} |C_G(x)| \chi(x^m)
= \chi(x^m). \qedhere
\]
\end{proof}


In the case of a prime $p$,
we will now first consider the more general situation of class functions $\chi^{(p)}$ and relate these to $p$-modular decomposition numbers
(or $p$-decomposition numbers, for short).
For more details on the background
of the modular theory, we refer the reader to
the textbooks \cite{CR, Navarro, Serre, Webb}.
For the convenience of readers less familiar with the modular theory,
and also to fix some notation on the way,
we give a brief introduction to the theory.

\color{black}  
For $g\in G$, we write $\ord (x)$ for the order of the element $x \in G$.  
We say that an element $x\in G$ is $p$-regular 
if   $p$ does not divide  $\ord (x)$.
 \color{black}  \color{black}

We start again in the general setting of a finite group
and recall some of the relevant notions.
For the connection between representations in characteristic~0 and prime characteristic~$p$, we require a $p$-modular splitting system $(R,F,K)$ where $R$ is a complete discrete valuation ring with quotient field $K=Q(R)$ of characteristic~0 and residue field $F$ of characteristic~$p>0$, such that
$K$ and $F$ are splitting fields for $G$ and its subgroups (e.g., $K,F$ are taken to be algebraically closed).
This is fixed for what follows, and $\Irr(G)$ now denotes the set of characters to the irreducible $KG$-representations.
Take $\chi \in \Irr(G)$, to an irreducible $KG$-module $M$, say;
then $M$ has as $R$-form, i.e.,
there is an irreducible $RG$-lattice, say $U$,
such that $M=K U$.
The composition factors of the $FG$-module
$\overline{U} := F\otimes_R U$
are uniquely determined by $\chi$.
When $S$ is a simple $FG$-module
and $\phi$ its (irreducible) Brauer character, then the decomposition number $d_{\chi\phi}$ is defined to be the multiplicity of $S$ as a composition factor of $\overline{U}$.
Let $\IBr(G)$ denote the set of irreducible Brauer characters; these correspond to the isomorphism classes of simple $FG$-modules.
Then
$$D:=(d_{\chi\phi})_{{\chi\in \Irr(G)}\atop {\phi\in \IBr(G)}}$$
is the $p$-modular decomposition matrix for $G$ (defined
as a matrix uniquely up to permutations of rows and columns).
The size of the matrix $D$ is given by the numbers
\[
|\Irr(G)| = k(G), \text{ the number of conjugacy classes of $G$},
\]
and
\[
|\IBr(G)| = k_p(G), \text{ the number of $p$-regular conjugacy classes of $G$}.
\]

Computing the decomposition matrix for a group is usually an  enormously  difficult problem;
  the recent use of geometry and categorical Lie theory in the study of decomposition matrices of  symmetric groups  was the topic of Geordie Williamson's   plenary talk at the 2018 ICM \cite{Willim}. 


By reordering the rows and columns,
the matrix $D$ can be put into a ``block diagonal form'',
where the blocks cannot be refined into a block diagonal sum
any further.
The $p$-blocks of $G$ correspond to the blocks in this
finest block diagonal form of~$D$. To a $p$-block $B$
associated to a block $D_B$ in the decomposition matrix $D$,
we associate all the irreducible characters and simple $FG$-modules
labelling the rows and columns of $D_B$; the
corresponding set of irreducible characters is denoted by $\Irr(B)$.
There is also an intrinsic definition of $p$-blocks as the indecomposable ideals
in the group algebra, and a direct criterion for when
two irreducible characters belong to the same $p$-block (see for example
\cite{CR, Feit, Navarro}).

While we have defined the decomposition matrix $D$ by rows,
the columns carry important information about projective modules.
Each simple $FG$-module $S$ has a unique (up to isomorphism) projective cover $P_S$; when $S$ runs through a system of representatives for the isomorphism classes of simple $FG$-modules, $P_S$ runs through a system of representatives for the isomorphism classes of indecomposable projective $FG$-modules.
Each projective $FG$-module $P_S$ can be lifted to a projective
$RG$-lattice,
which has a corresponding character $\Phi_S$ of $G$ (over $K$);
when $\phi$ is the Brauer character to $S$, we also write  $\Phi_\phi$ for this character.
The character $\Phi_\phi$ is uniquely determined by the simple module $S$,
and we have the following property of the columns of the decomposition matrix:
\[
\Phi_\phi = \sum_{\chi\in \Irr(G)} d_{\chi \phi} \chi \;,
\text{ for } \phi \in\IBr(G).
\]
The character of a projective (indecomposable) $RG$-lattice is also called a ($p$-)projective (indecomposable) character of $G$. The elements in the $\Z$-span of the characters $\Phi_\phi$, $\phi\in \IBr(G)$, are called virtual ($p$-)projective characters.
In the character ring
\[
R_\Z (G) =
\{ \textstyle \sum_{\chi \in \Irr(G)} a_\chi \chi \mid \text{  all } a_\chi \in \Z \}
\]
of virtual characters, the virtual projective characters
are characterized by the property that they vanish on all $p$-singular elements of $G$, i.e., on the set
\[G_p = \{ x \in G \;\big|\; p \text{ divides } \ord (x)\}.\]

Now we are ready for the following results (we fix the $p$-modular system $(R,F,K)$). First an easy observation for which we will see a rather direct application in the context of symmetric groups further below, in Corollary~\ref{cor:rho}.

\begin{lem}\label{lem:basicobs}
Let $G$ be a finite group, $\chi$ a character of $G$ such that
$\chi(g^p)=\chi(g)$ for all $p$-regular elements $g\in G$. 
Let $\vartheta $ be a class function on $G$ that vanishes on $G_p$.
Then we have
\[
\gen{\chi^{(p)},\vartheta } = \gen{\chi, \vartheta }.
\]
In particular, when $\chi,\vartheta $ are in addition irreducible, then
\[
\gen{\chi^{(p)},\vartheta } = \delta_{\chi\, \vartheta }.
\]
\end{lem}

\begin{proof}
The claim follows immediately by observing that
$\gen{\chi^{(p)}-\chi,\vartheta } =0$,
since the two class functions $\chi^{(p)}-\chi$ and $\vartheta $ vanish
on complementary sets.
\end{proof}

Using this, we obtain the following information on decomposition numbers:
\begin{prop}\label{prop:basicobs}
Let $G$ be a finite group, $\chi\in \Irr(G)$ such that
$\chi(g^p)=\chi(g)$ for all $p$-regular elements $g\in G $.
For $\psi \in \Irr(G)$ set $a_\psi = \gen{\chi^{(p)}, \psi}$.
Then for all $\phi \in \IBr(G)$ we have
\[
d_{\chi \phi}
= \sum_{\psi \in \Irr(G)} a_\psi d_{\psi \phi}
.\]
\end{prop}

\begin{proof}
As noted earlier, the projective character $\Phi_\phi$ to $\phi \in \IBr(G)$
vanishes on $G_p$. Hence using Lemma~\ref{lem:basicobs} we compute
\[
\dstyle
d_{\chi\phi}
=
\gen{\chi, \Phi_\phi} 
=
\gen{\chi^{(p)},\Phi_\phi} 
=  \sum_{\psi \in \Irr(G)} a_\psi \gen{\psi, \Phi_\phi} 
=  \sum_{\psi \in \Irr(G)} a_\psi d_{\psi \phi}
.
\qedhere
\]
\end{proof}

\begin{rems}\label{rem:applicationtoproj}
{\rm
(1)
First, some comments on the assumption on $\chi$ in Lemma~\ref{lem:basicobs} and Proposition~\ref{prop:basicobs}.

Clearly, when $G$ is a finite group such that
each $p$-regular element $g\in G$
is conjugate to its power $g^p$, then
the assumption on $\chi$ in these results
is satisfied for all characters of $G$.
In particular, this holds for the symmetric groups at all primes~$p$:
when $g\in S_n$ is $p$-regular,
$p$ does not divide any part of its cycle type $\la$,
and so $g^p$ has the same cycle type~$\la$.

More generally, if $G$ is a group with a rational character table,
then $G$ has this property for all primes $p$,
since each element $g\in G$
is conjugate to all powers $g^m$, where $\gcd(m, \ord g)=1$.

But of course, there are many examples of groups where this property
holds only for some primes $p$,
and examples where it does not hold for $G$ at a given prime $p$,
but for certain irreducible characters.

(2)
We emphasise that in applying Proposition~\ref{prop:basicobs}, the decomposition of $\chi^{(p)}$ into irreducible characters
(which is a characteristic~0 computation) gives a linear relation
between the $p$-modular decomposition numbers
$d_{\psi\phi}$, $\phi \in \IBr(G)$, which
are in general hard to determine.

(3)
We also emphasise that Lemma~\ref{lem:basicobs} may not only be applied towards the
decomposition of the indecomposable projective characters $\Phi_\phi$,
but for obtaining linear relations between the coefficients of arbitrary projective characters.
}
\end{rems}

We want to mention a method for obtaining
a suitable class function
that can easily be applied in particular in the case of the symmetric groups.
First we state this property, called \emph{Block orthogonality} (\cite{Navarro})
for general finite groups.
 
\begin{prop}\label{prop:blockorth} 
Let $B$ be a $p$-block of the finite group $G$, and let $x\in G$ be $p$-regular.
Then for all $p$-singular $y\in G$ we have
\[
\sum_{\psi \in \Irr(B)} \psi(x) \psi(y) =0.
\]
\end{prop}

We formulate a consequence for square splitting, using the case  $p=2$;
note that the class function $\vartheta $ appearing below
is a virtual 2-projective character
in the case of $G=S_n$.

\begin{cor}\label{cor:block-part}
Let $\chi \in \Irr(G)$ such that $\chi(g)=\chi(g^2)$ for all $2$-regular
$g\in G$. Let $B$ be a 2-block of $G$, let $x\in G$ be 2-regular
and set $\vartheta =\sum_{\psi \in \Irr(B)} \psi(x^{-1}) \psi$.
Then
\[
\gen{S^2(\chi),\vartheta }- \gen{A^2(\chi),\vartheta }
=
\left\{
\begin{array}{cl}
\chi(x) & \text{if } \chi \in \Irr(B)\\
0 & \text{otherwise}
\end{array}
\right.
\]
\end{cor}

Turning to the case of the symmetric groups, we observe the following.  When we have computed the coefficients $g(\la,\la,\mu)$
for the Kronecker square $\cla^2$,
then from the splitting of the square
into its symmetric and alternating part, we obtain
the explicit decomposition of $S^2(\cla)-A^2(\cla)=\cla^{(2)}$.
By Proposition~\ref{prop:basicobs},
this implies linear relations for the 2-modular
decomposition numbers.
For the symmetric group $S_n$, the (isomorphism classes of) simple modules (and their Brauer characters) in characteristic~2 are labelled by the 2-regular partitions of $n$, i.e., the partitions of $n$ into distinct parts.
The decomposition numbers are then written as
$d_{\la \mu}$, where $\la$ is the partition label of the irreducible complex
character and $\mu$ is the 2-regular partition labelling the simple module
$D^\mu$, its projective cover $P^\mu$ and corresponding projective
character $\Phi^\mu$.
Thus, 
$$\Phi^\mu = \sum_{\la \in P(n)} d_{\la \mu} \cla .$$

The observation made above on the 2-decomposition numbers adds to the motivation for determining the splitting of the Kronecker squares, for which so far very few results are known.
We will mention some such results on square splittings
in the following sections as well as provide
new results (and also conjectures).

One strategy in the following will be to pick suitable
characters $\Phi$ which are virtual projective characters
(or close to such characters), and apply Lemma~\ref{lem:basicobs}
(or a variation thereof).  
We consider some examples for this strategy.

\begin{theorem}\label{cor:rho}
Let $k\in \N$, $n=k(k+1)/2$ and $\rho_k=(k,k-1,\ldots,1)\in P(n)$ the staircase partition.
Let $\la \in P(n)$. Then we have
\[
sg(\la,\rho_k) - ag(\la,\rho_k)=
\gen{S^2(\cla),[\rho_k]} - \gen{A^2(\cla),[\rho_k]}
=
\left\{
\begin{array}{cl}
1 & \text{if } \la=\rho_k \\
0 & \text{if } \la\ne \rho_k
\end{array}.
\right.
\]
In particular, $[\rho_k]$ is a constituent of $S^2([\rho_k])$,
and for any $\la\ne \rho_k$, the Kronecker coefficient
$g(\la,\la,\rho_k)=2\, sg(\la,\rho_k)$ is even.
\end{theorem}

\begin{proof}
Since $\rho_k$ is a 2-core (i.e., a partition without a 2-hook),
the character $[\rho_k]$ itself is a 2-projective irreducible character
(see \cite{JK}).
Furthermore, any 2-regular element $g\in S_n$ is conjugate to its square~$g^2$.
Hence we can apply Lemma~\ref{lem:basicobs} to
$\vartheta =[\rho_k]$, $\chi=\cla$ and $p=2$, and we obtain
the result immediately.
\end{proof}

\begin{rem}{\rm
The final assertion implies in particular that $[\rho_k]$ is a constituent of $[\rho_k]^2$, which is a special case of the fact that $g(\la,\la,\la)>0$ for all symmetric partitions $\la\in P(n)$ (see \cite{BB}).
}\end{rem}

We have mentioned earlier a method for finding
$p$-projective characters $\Phi$ in particular for the symmetric groups,
via a weighted sum of characters in a $p$-block.
Taking $x=1$ in Corollary~\ref{cor:block-part}
we see that for a 2-block $B$
the contributions of the weighted $B$-part in $S^2(\chi)$
and that of the weighted $B$-part in $A^2(\chi)$
coincide if $\chi$ does not belong to $B$ and
differs by $\chi(1)$ if it does.
In the case of the character $[\rho_k]$, this was
particularly simple, as this is the unique irreducible
character in its 2-block, so we could drop the weight.

We state the consequence of
Corollary~\ref{cor:block-part} in the case of $S_n$
explicitly,
and illustrate this with examples
below.
Fortunately, for the symmetric groups
we have a combinatorial criterion that determines
when two irreducible characters belong to the same $p$-block:
this is the case if and only if their labelling partitions
have the same $p$-core (see \cite{JK}).
Thus a 2-block $B$ of $S_n$ is combinatorially
determined by the common 2-core $\la_{(2)}$ of
all the partitions $\la$ labelling the irreducible
characters in $\Irr(B)$. The common 2-weight
of all these characters $\la$ is
called the 2-weight $w=w(B)$ of $B$.
Thus $n=2w+|\rho|$.

\begin{cor}\label{cor:2blockS_n}
Let $n\in \N$, $\la\in P(n)$. Let $B$ be a 2-block of $S_n$,
with associated 2-core $\rho$.
Then we have
\[
\sum_{\mu \in P(n), \mu_{(2)} = \rho } (sg(\la,\mu) - ag(\la,\mu)) f(\mu)  
=
\left\{
\begin{array}{cl}
f(\la)    & \text{if } \la_{(2)}=\rho \\
0 & \text{otherwise.}
\end{array}
\right.
\]
\end{cor}

\begin{example}{\rm
Let $\la=(3,2^2)$. Then
\[
\begin{array}{rcl}
S^2(\cla) &=&
[7]+[6,1]+2[5,2]+[4,3]+2[4,2,1]+2[3,2^2] 
+[3,2,1^2]+[3,1^4]+[2^3,1]+[2,1^5]\\[5pt]
A^2(\cla) &=& 
\color{black} [5,1^2]+ [4,2,1]  + 2 [4,1^3]  +[3^2,1] + 2[3,2,1^2] 
+ [3,1^4]  
+ [2^2,1^3]   
\end{array}
\]
Here, $\la_{(2)}=(1)$, i.e., $\cla$ belongs to the so-called principal 2-block $B_0$.
Furthermore, $f(\la)  = 21$.
%
%
%
The constituents in $S^2(\cla)$ and $A^2(\cla)$ belonging to $B_0$ are
$$[7],[5,2],[4,2,1],[3,2^2],[3,2,1^2],[3,1^4]$$
and
$$[5,1^2],[4,2,1],[3^2,1],[3,2,1^2],[3,1^4],[2^2,1^3],$$
respectively.
As expected, the weighted sum on the left hand-side above is then
$$
1+2\cdot 14 -15 + 35 - 21 + 2\cdot 21 - 35 - 14 = 21
.$$
}
\end{example}


We also want to extend the application of these methods in another direction
and look at 2-blocks of small weight.

\begin{cor}\label{cor:2blockS_n-smallweight}
Let $n\in \N$, $\la\in P(n)$. Let $B$ be a 2-block of $S_n$,
with associated 2-core $\rho_k=(k,\ldots,2,1)$.
\begin{enumerate}
\item
Assume $w(B)=1$.
Set $\tau_k=\rho_k + (2)$. 
Then we have
\[
\gen{S^2(\cla),[\tau_k]+[\tau_k^t]}
- \gen{A^2(\cla),[\tau_k]+[\tau_k^t]}
=
\left\{
\begin{array}{cl}
1 & \text{if } \la \in \{\tau_k,\tau_k^t \}; \\
0 & \text{otherwise. } 
\end{array} 
\right.
\]
\item
Assume $w(B)=2$.
Set $\zeta_k=\rho_k + (4)$, 
$\xi_k=\tau_k \cup (1^2)$.  
Then we have
\[
\gen{S^2(\cla),[\zeta_k]+[\xi_k]+[\zeta_k^t]}
- \gen{A^2(\cla),[\zeta_k]+[\xi_k]+[\zeta_k^t]}
=
\left\{
\begin{array}{cl}
1 & \text{if } \la \in \{\zeta_k,\xi_k,\zeta_k^t\}; \\
0 & \text{otherwise. }.
\end{array}
\right.
\]
\end{enumerate}
\end{cor}

\begin{proof}
(1) When $w(B)=1$,
we have $\Irr(B)=\{[\tau_k],[\tau_k^t]\}$; note that these two characters
have the same values on all 2-regular elements (which belong to $A_n$).
Then $\Phi=[\tau_k]+[\tau_k^t]$ vanishes on $G_2$, and
the result follows (using Corollary~\ref{cor:block-part}).

(2) When $w(B)=2$,
the  irreducible characters in $B$ are
the ones already listed, together with the ones to $\rho_k+(2^2)$
and its conjugate, i.e.,
$[k+2,k+1,k-2\ldots,2,1]$ and $[k,k-1,\ldots,3,2^3,1]$.
\color{black} 
Now let $\mu=H(\zeta_k)$ be the partition formed from the principal hook lengths
of~$\zeta_k$.  \color{black}
\color{black}
By the Murnaghan-Nakayama formula,
the characters to $\zeta_k,\xi_k,\zeta_k^t$
all have the same value on the class of type $\mu$
(which is $\pm 1$),
while the further two characters are zero on this class.
Hence taking
$\Phi=[\zeta_k]+[\xi_k]+[\zeta_k^t]$
implies the claim (again using Corollary~\ref{cor:block-part}).
\end{proof}

  \section{Splitting the square: Hooks}\label{sec:hooks}

In this and the following sections we will discuss a number of cases  for which a formula for  some family of Kronecker coefficients is known, 
and we want to obtain   refined
information on the decomposition  into its symmetric and alternating parts.
We will also recall some of the few results where
the splitting has already been determined.

   Kronecker products of characters involving (various combinations of) hooks and 2-part partitions
formulae were first given by Remmel; in the case of products of hook characters a small error occurred in \cite{Remmel-hook}, but in the case of products of 2-part partitions, the formulae  in \cite{RW-2part}
contained multiple errors.
Alternative correct formulae were later provided by Rosas \cite{Rosas};
  these are not in all cases manifestly positive, but they do still satisfy the ``taste test" of a ``combinatorial solution".
Later, manifestly positive  formulae for special products of characters of 2-part partitions in \cite{GWXZ,BWZ10,Man10}.
More recently, general
formulae were obtained in the cases when one factor is labelled by special 2-part partitions \cite{BaOr} or when one factor is a hook \cite{Blasiak, Liu}.
We will discuss some interesting cases of squares of 2-part partitions
in Section~\ref{sec:2part}.
 
\subsection{Splitting squares of hooks}\label{ssec:hooksquares}

In the case of hooks, the splitting of the squares has recently been
determined by M\'esz\'aros and Wolosz \cite{MW}. 
We recall their result here:

\begin{theorem}\label{thm:MW} \cite{MW}
Let $n\in \N$. 
Let $\la=(n-k,1^k)$ be a hook partition.
\begin{enumerate}
\item[{(1)}]
Let $\mu=(n-m,1^m)$ be a hook with $0\le m \le 2\min(k,n-k-1)$.
Then
\[
\begin{array}{cl}
sg(\la,\mu)=1 &\text{if } m\equiv 0 \text{ or } 1 \mod 4 ,\\
ag(\la,\mu)=1 &\text{if } m\equiv 2 \text{ or } 3 \mod 4 .
\end{array}
\]

\item[{(2)}]
Let $\mu$ be a double-hook 
$(\mu_1,\mu_2,2^{d_2},1^{d_1})$,
where $\mu_1\ge \mu_2\ge 2$, $d_2,d_1\ge 0$.
Then
\[
\begin{array}{rl}
sg(\la,\mu)=&
\left\{
\begin{array}{ccll}
2 & \text{if}& |2k+1-n|\le \mu_1-\mu_2 &\text{ and } d_1\equiv 0 \mod 4 ;\\
1 & \text{if}& |2k+1-n|\le \mu_1-\mu_2 &\text{ and $d_1$ odd},\\
 \color{black}1& \text{or}& |2k+1-n|= \mu_1-\mu_2+1 &\text{ and } d_1\equiv 0 \mod 4 .\\
\end{array}
\right.
\\[7pt]
ag(\la,\mu)
=&
\left\{
\begin{array}{ccll}
2 & \text{if}& |2k+1-n|\le \mu_1-\mu_2 &\text{ and } d_1\equiv 2 \mod 4 ;\\
1 & \text{if}& |2k+1-n|\le \mu_1-\mu_2 &\text{ and $d_1$ odd},\\
\color{black}1 & \text{or}& |2k+1-n|= \mu_1-\mu_2+1 &\text{ and } d_1\equiv 2 \mod 4 .\\
\end{array}
\right.
\end{array}
\]
\end{enumerate}
 For all other partitions $\mu\in P(n)$,
 the coefficients $sg(\la,\mu)$ and $ag(\la,\mu)$ are~0.
\end{theorem}


\subsection{Hook constituents}\label{hook-constituents}

We now want to take up the ideas of Section~\ref{sec:rel-to-dec}
to obtain some information on the distribution of hook constituents in the
symmetric and alternating part of an arbitrary Kronecker square.
For an application of Lemma~\ref{lem:basicobs} we take a closer look at the character of $S_n$ obtained by summing over all hook characters:
\[
\chi_{\text{hook}} = \sum_{k=0}^{n-1} [n-k,1^k]
.
\]
This character has already been fruitfully used in \cite{B-spinSaxl} for finding further constituents in the Saxl square~$[\rho_k]^2$.

The character $\chi_{\text{hook}}$ is a virtual 2-projective character; more precisely,
in recent years it was shown (see \cite{B-spinSaxl, R-hook, T})
that its values are as follows.
\\
Let $\sigma_\al\in S_n$ be an element of cycle type $\alpha$.
Then
\[
\chi_{\text{hook}}(\sigma_\al) =
\left\{ \begin{array}{ll}
2^{\ell(\alpha)-1} & \text{if $\alpha$ is 2-regular} \\
0 & \text{otherwise}
\end{array}\right. \:.
\]


The following result tells us that
the symmetric part
$S^2(\cla)$ and the alternating part
$A^2(\cla)$
always have the same number of hook constituents (counted with multiplicities)
if $\la$ is not a hook,
and $S^2(\cla)$ contains one more hook constituent if~$\la$ is a hook.

\begin{prop}\label{prop:hooks}
Let $\la\in P(n)$.
Then we have
\[
\gen{\cla^{(2)},\chi_{\text{hook}}} =
\left\{
\begin{array}{cl}
1 & \text{if  $\la$ is a hook} \\
0 & \text{otherwise }
\end{array}.
\right.
\]
or equivalently,
\[
\sum_{k=0}^{n-1} sg(\la,[n-k,1^k])
- \sum_{k=0}^{n-1} ag(\la,[n-k,1^k])
=
\left\{
\begin{array}{cl}
1 & \text{if  $\la$ is a hook} \\
0 & \text{otherwise }
\end{array}.
\right.
\]
%
In particular, when $\la$ is not a hook,
$A^2(\cla)$ must always contain a hook constituent.
\end{prop}

\begin{proof}
We apply Lemma~\ref{lem:basicobs}
with $\chi=\cla$, $p=2$, and $\vartheta =\chi_{\text{hook}}$.
Thus we obtain
\begin{align*}\textstyle
 \sum_{k=0}^{n-1} sg(\la,[n-k,1^k])
- \sum_{k=0}^{n-1} ag(\la,[n-k,1^k])
&=
\gen{S^2(\cla)-A^2(\cla), \chi_{\text{hook}}}
\\&
=
\gen{\cla^{(2)},\chi_{\text{hook}}}\\
 &=
\gen{\cla,\chi_{\text{hook}}}
\\ &
=
\left\{
\begin{array}{cl}
1 & \text{if  $\la$ is a hook} \\
0 & \text{otherwise }
\end{array}.
\right.
\end{align*}The final assertion follows since we have already seen that $S^2(\cla)$ always contains the hook constituent~$[n]$.
\end{proof}


We want to make this result more precise and at the same
time illustrate how to use the idea of Lemma~\ref{lem:basicobs}
when we consider a class function $\vartheta $ close to a virtual projective character.
The important aspect is that one should have good control over the
non-zero values on 2-singular classes.  
%
%
%
%
%
%
%
For $n\in\N$, we set
\[
\vartheta _n = \sum_{k=0}^{n-1} (-1)^k [n-k,1^k].
\]
This is just a special case of the class function $\vartheta _C$
defined earlier, here for the class $C$ of $n$-cycles in~$S_n$.
Thus, for $\sigma_\al\in S_n$ of cycle type $\al$, we have
\[
\vartheta _n(\sigma_\al) =
\left\{
\begin{array}{cl}
n & \text{if  } \al = (n)  \\
0 & \text{otherwise }
\end{array}.
\right.
\]
Now we consider
\[
\Phi_n
=
\sum_{m=0}^{\lfloor (n-1)/2 \rfloor} [n-2m,1^{2m}]
,
\]
the sum of the characters to hooks of even leg length.
\\
By the discussion above,
\[\Phi_n = \tfrac 12 (\chi_{\text{hook}} + \vartheta _n ) \]
is a virtual 2-projective
character when $n$ is odd; when $n$ is even, the
only 2-singular elements where $\Phi_n$ does not vanish,
are the $n$-cycles.
We now want to use $\Phi_n$
to obtain information on the constituents to hooks of even
leg length in $\cla^{(2)}$. 

\begin{prop}\label{prop:evenhooks}
Let $n\in \N$ and $\la\in P(n)$; set $\chi=\cla$ and let $\Phi_n$ be as above.
When $n$ is odd, we have
\[
\gen{\chi_S-\chi_A,\Phi_n} =
\left\{
\begin{array}{cl}
\frac 12 (1 + (-1)^k) & \text{if  } \la = (n-k,1^k) \\
0 & \text{otherwise.}
\end{array}
\right.
\]
When $n$ is even, we have
\[
\gen{\chi_S-\chi_A,\Phi_n} =
\left\{
\begin{array}{ll}
\frac 12 (1 + (-1)^k) & \text{if  } \la = (n-k,1^k) \text{ and } n-k>k \\[7pt]
\frac 12 (1 + (-1)^{k-1}) & \text{if  } \la = (n-k,1^k) \text{ and } n-k \le k \\[7pt]
\frac 12 \cla \bigl((\frac n2 , \frac n2 )\bigr) & \text{otherwise.}
\end{array}
\right.
\]
In any case, we have
\[
\gen{\chi_S-\chi_A,\Phi_n} \in \{-1,0,1\},
\]
i.e.,  the numbers of constituents (counted with multiplicity)
in $\chi_S$ and $\chi_A$
to hooks of even leg length
differ at most by one.
\end{prop}

\begin{proof}
We have
\[
\gen{\chi_S-\chi_A,\Phi_n} = \gen{\chi^{(2)}, \Phi_n}
=
\tfrac 12 \gen{\chi^{(2)}, \chi_{\text{hook}}}
+\tfrac 12 \gen{\chi^{(2)}, \vartheta _n}.
\]
By Proposition~\ref{prop:hooks}, the first summand
contributes zero if $\la$ is not a hook,
and $\frac 12$ otherwise.

When $n$ is odd, $\vartheta _n$ vanishes on $G_2$ 
and is non-zero only on $n$-cycles.
By Lemma~\ref{lem:theta_C}
the second summand on the right hand side above
is then $\frac 12 \chi(\sigma_{(n)}^2)=\frac 12 \cla((n))$;
hence this is zero if $\la$ is not a hook, and
it is $\frac 12 (-1)^k$ when $\la=(n-k,1^k)$.
This yields the assertion in this case.

Now assume that $n$ is even. Then Lemma~\ref{lem:theta_C} gives
\[
\tfrac 12 \gen{\chi^{(2)}, \vartheta _n}
=
\tfrac 12 \chi(\sigma_{(n)}^2)
=
\tfrac 12 \cla\bigl((\tfrac n2 , \tfrac n2)\bigr).
\]
If $\la$ is not a hook, this immediately yields the claim.
\\
Now assume that $\la=(n-k,1^k)$.
When $n-k>k$, we have
\[
[n-k,1^k]\bigl((\tfrac n2 , \tfrac n2)\bigr)
=
[\tfrac n2 -k,1^k] \bigl((\tfrac n2 )\bigr)
=
(-1)^k.
\]
When $n-k\le k$, we have
\[
[n-k,1^k]\bigl((\tfrac n2 , \tfrac n2)\bigr)
=
(-1)^{\tfrac n2 -1}[n -k,1^{k-\frac n2}] \bigl((\tfrac n2 )\bigr)
=
(-1)^{\frac n2 -1-k+\frac n2} = (-1)^{k-1}.
\]
Thus also in the hook case we arrive at the stated formulae.


For the final assertion, we only have to consider the case where
$n$ is even and $\la$ is a non-hook partition.
By the Murnaghan-Nakayama formula, the only non-hook partitions $\la$
with $\cla((\frac n2 , \frac n2))\ne 0$ are partitions $\la$
with $dl(\la)=2$ 
and of $\frac n2$-weight exactly~2.
Furthermore, the only possible non-zero values are then $\pm 2$.
Thus also in this case $\gen{\chi_S-\chi_A,\Phi_n}$ is $\pm 1$, as claimed.
\end{proof}


\begin{rem}{\rm
In the case of even $n$, it is not difficult to count (and construct) all $\la$
with $\cla((\frac n2 , \frac n2))\ne 0$.
We already know that we only have to consider $\la$ with
$dl(\la)\le 2$ and of $\frac n2$-weight exactly~2.
All $n$ hooks $\la$ satisfy $\cla((\frac n2 , \frac n2))= \pm 1$.
As the centraliser of an element of cycle type $(\frac n2 , \frac n2)$
has order $\frac{n^2}2$, there are then exactly $\frac 18 n(n-2)$ non-hook partitions
$\la$ that contribute a value $\pm 2$ on this class.
}\end{rem}


\color{black}
Recall that a hook $(a,1^b)$  is said to have ladder length $b$ and arm length $a-1$ for $a> 0 $ and $b\geq0$.  \color{black}
From Proposition~\ref{prop:hooks} and Proposition~\ref{prop:evenhooks}
we now obtain the following result that says that
the sum of the multiplicities of the constituents
to hooks of even leg length (and odd leg length, respectively) in
$S^2(\cla)$ and $A^2(\cla)$ coincide.  

\begin{cor}\label{cor:samehooksums}
Let $\la\in P(n)$. Assume that $\la$ is not a hook,
and that $\la$ is also not a double-hook
with $\cla((\frac n2 , \frac n2))\ne 0$ (when $n$ is even).
Then
\[
\begin{array}{rcl}
\dstyle
\sum_{m=0}^{\lfloor (n-1)/2 \rfloor} sg(\la,(n-2m,1^{2m}))
&=&
\dstyle
\sum_{m=0}^{\lfloor (n-1)/2 \rfloor} ag(\la,(n-2m,1^{2m}))
\\[5pt]
\dstyle
\sum_{m=1}^{\lfloor n/2 \rfloor} sg(\la,(n-2m+1,1^{2m-1}))
&=&
\dstyle
\sum_{m=1}^{\lfloor n/2 \rfloor} ag(\la,(n-2m+1,1^{2m-1})) .
\end{array}
\]
\end{cor}


\begin{rem}
{\rm
Towards the Saxl Conjecture, it has been proved with a variety of methods
that all hook characters $[n-k,1^k]$ are constituents of $[\rho_k]^2$
(see \cite{B-spinSaxl, I-2015, PPV}).

From
the results above we deduce
a refinement on the distribution of the hooks into the symmetric and alternating part
of $[\rho_k]^2$.
For $k\ge 5$, the sum of the multiplicities of the constituents
to hooks of even leg length (and odd leg length, respectively) in
$S^2([\rho_k])$ and $A^2([\rho_k])$ coincide.
For $k=3$ and~$4$, $[\rho_k]((\frac n2, \frac n2 ))=-2$, and
the sum of the multiplicities of the characters to hooks of odd leg length is one larger in $S^2([\rho_k])$ than in $A^2([\rho_k])$, and conversely for the sums to the even hook length characters.
%
%
%
%
%
%
More precisely, for $k=4$
the part in $S^2([\rho_4])$ corresponding to hooks of even leg length is
$$[10]+3[8,1^2]+12[6,1^4]+10[4,1^6]+[2,1^8],$$
\color{black} (note that the coefficients sum to 27) \color{black} 
and the corresponding part in $A^2([\rho_4])$ is
$$6[8,1^2]+12[6,1^4]+8[4,1^6]+2[2,1^8] $$
\color{black}  (note that the coefficients sum to 28).   \color{black} 
}\end{rem}

  \section{Splitting of Kronecker squares: Small depth}\label{sec:splitsquare-smalldepth}

In this section we start by studying
constituents of small depth in arbitrary Kronecker
squares,
and we also consider squares $\cla^2$ where
$\la$ is a partition of
small depth.
 
\subsection{Constituents of small depth}

First we consider the constituents of small depth in
Kronecker squares.
We recall the following information on the multiplicities
of constituents up to depth~3 explicitly (see \cite{Saxl, V, Z}).
We see that already the formulae for constituents of depth~3
get involved;
in fact, also formulae for the case of depth~4 constituents have been determined
by Vallejo \cite{V}.

\begin{prop}\label{prop:smalldepthconstituents}
Let $\la \in P(n)$, $\la \neq (n), (1^n)$.
Let $r_k=\#\{k\text{-hooks in } \lambda\} $ for $k=1,2,3$,
and let
$r_{21}=\#\{\text{non-linear 3-hooks $H$ in } \lambda   \}$. Then
$$
[\lambda]^2  =  [n]+a_1[n-1,1] +a_2[n-2,2] +b_2[n-2,1^2]
  +a_3[n-3,3] +b_3[n-3,1^3]+c_3[n-3,2,1]  \\[5pt]
 + \dots 
$$
where
$a_1=  r_1-1$, $a_2=  r_2+r_1(r_1-2)$ for $n\geq 4$,
$b_2=(r_1-1)^2$,
\\[5pt]
$a_3=  r_1(r_1-1)(r_1-3)+r_2(2r_1-3)+r_3$,
for $n \geq 6$,\\[5pt]
$b_3=  r_1(r_1-1)(r_1-3)+(r_1-1)(r_2+1)+r_{21}$,
for $n \geq 4$, \\[5pt]
$c_3= 2r_1(r_1-1)(r_1-3)+r_2(3r_1-4)+r_1+r_{21}$,
for $n \geq 5$.\\[1ex]
 In particular, for $n \ge 4$ we have $a_2 >0$.
\end{prop}

In the following result we provide the splitting into the two parts
of a square for the constituents up to depth~2.

\begin{theorem}\label{thm:smalldepthconst}
Let $\la \in P(n)$, $n\ge 4$. 
Let
$a_1$,
$a_2$, 
$b_2$ 
and $r_1$ be as in Proposition~\ref{prop:smalldepthconstituents}.
Let
$a_{1,S}=sg(\la,(n-1,1))$,
$a_{1,A}=ag(\la,(n-1,1))$,
and similarly define the splitting of the coefficients $a_2$ and $b_2$
into $a_{2,S}$, $a_{2,A}$,
and $b_{2,S}$, $b_{2,A}$, respectively.
\\
Then
\[
a_{1,S}=a_1, a_{1,A}=0, \;
a_{2,S}=a_2, a_{2,A}=0, \;
b_{2,S}={r_1-1 \choose 2}, b_{2,A}={r_1 \choose 2},
\]
i.e., with coefficients as determined above,
\[
S^2(\cla) = [n]+a_1[n-1,1]+a_2[n-2,2] + b_{2,S}[n-2,1^2] +
\text{constituents of depth $>2$}
\]
and
\[
A^2(\cla) = b_{2,A}[n-2,1^2] +
\text{constituents of depth $>2$}.
\]
\end{theorem}

\begin{proof}
Let $n \ge 3$, and let $\la \in P(n)$.
We have already seen earlier that the constituent $[n]$
only appears once in $\cla^2$, and it is located
in the symmetric part.

We now show that $[n-1,1]$ never appears in $A^2(\cla)$.
Assume it does occur; then $[n-1]$ appears in the restriction
\[
A^2(\cla){\downarrow}_{S_{n-1}} =
\sum_{B} A^2([\la_B]) + \sum_{B\ne C} [\la_B]  [\la_C],
\]
where $B$ and $C$ run over all removable boxes of $\la$,
and $\la_D$ denotes the partition of $n-1$ obtained by
removing a (corner) box $D$ from $\la$.
But we already know that no summand $A^2([\la_B])$ contains
a constituent $[n-1]$.
Furthermore, since in the second sum $\la_B\ne \la_C$,
no product $[\la_B]  [\la_C]$ contains $[n-1]$.   
Hence $a_{1,A}=0$ and $a_{1,S}=a_1$.

Now let $n \ge 4$.  
We want to consider $[n-2,2]$ and  $[n-2,1^2]$ in $X^2(\cla)$ for $X=S$ or $A$.
This time we consider the restriction
of $X^2(\cla)$ to $S_{n-2}\times S_2$,
and we apply Lemma~\ref{lem:SA-prodgroups};
notice that here the square of the linear character of $S_2$
is always trivial.
Then  for $X \in \{S,A\}$  we obtain
\[
 X^2(\cla){\downarrow}_{S_{n-2}\times S_2}
 = 
\dstyle
\sum_{
\begin{subarray}c
\mu \in P(n-2)\\ \mu \subseteq \la
\end{subarray}}
2^{cc(\la/\mu)-1} X^2([\mu]) \times [2]
+\dstyle \sum_
{
\begin{subarray}c
\phantom{\ne}(\mu^1,\nu^1) \phantom{\ne}
\\
\ne (\mu^2,\nu^2)\phantom{\ne}
\end{subarray}
}
c_{\mu^1,\nu^1}^\la c_{\mu^2,\nu^2}^\la
[\mu^1][\mu^2] \times [\nu^1][\nu^2],
 \]
where $cc(\la/\mu)$ is the number of connected components of the diagram $\la/\mu$, $(\mu^j,\nu^j)\in P(n-2)\times P(2)$ and
$c_{\mu^j,\nu^j}^\la$ are the Littlewood-Richardson coefficients.

We want to show that $[n-2,2]$ never appears in $A^2(\cla)$.
We notice that  $[n-2,2]$ has a constituent $[n-2]\times [2]$ in its restriction
to $S_{n-2}\times S_2$.
This cannot occur in the first summand of $A^2([\la])$,  because  $A^2([\mu])$ does not
contain $[n-2]$.
In the second sum it could only occur for $\mu^1=\mu^2$, to have $[n-2]$ in the  first component; but then we must have $\nu^1 \ne \nu^2$, so that
then $[\nu^1][\nu^2]=[1^2]$, and thus we do not get
$[n-2]\times [2]$.
Hence $a_{2,A}=0$ and $a_{2,S}=a_2$.

Still assuming  $n\ge 4$,
we finally consider the distribution of the
constituents $[n-2,1^2]$.
As above, we consider the restriction of both $S^2([\la])$ and $A^2([\la])$ to $S_{n-2}\times S_2$ and we notice that the second term is the same in both cases.  
Now, the term $[n-2]\times [1^2]$ in the restriction
$X^2([\la])$ can only appear in this second summand (the first summand consists only of characters of the form $[\nu]\times [2]$ for some $\nu \in P_2(n)$); moreover, it can only come from
the constituents $[n-1,1]$ and $[n-2,1^2]$, where it appears once in the restriction.  
Hence we deduce that $a_1+b_{2,S}=b_{2,A}$, i.e., $b_{2,A}-b_{2,S}=r_1-1$.
On the other hand, by Proposition~\ref{prop:smalldepthconstituents}
we have $b_{2,A}+b_{2,S}=b_2=(r_1-1)^2$, hence $b_{2,A}=\frac 12 r_1(r_1-1)$ and $b_{2,S}=\frac 12 (r_1-2)(r_1-1)$, as claimed.
\end{proof}

\begin{prop}\label{rectangleresult}
For $a,b\geq 3$ we have that 
\begin{align*}
S^2([a^b]) &= [ab] + [ab-2,2] 
+ [ab-3,3] 
+\dots 
\qquad\ 
A^2([a^b])  = 
  [ab-3,1^3] 
+\dots
\end{align*}
where the $\dots$ are terms of depth 4 or higher.  
\end{prop}

\begin{proof}
 The terms of depth 2 and the term $(ab-3,2,1)$ can   be obtained from \Cref{prop:smalldepthconstituents,thm:smalldepthconst}. 
We now check the remaining constituents.  
For $X\in \{S,A\} $ we   consider the   $[ab-3] \times [3]$ and 
$[ab-3] \times [1^3]$ isotypic summands of    the following
\begin{equation}\label{forlaters}
X^2(\la){{\downarrow}}_{S_{n-k}\times S_k} = 
\sum_{\mu \in P_{(a^b)}(k)} X^2([(a^b)-\mu] \times   [\mu] )
+ \sum_{
\begin{subarray}c
\mu,\nu \in P_{(a^b)}(k) \\
\mu\neq \nu
\end{subarray}} 
 [(a^b)-\mu] [(a^b)-\nu]    \times [\mu] [\nu].
\end{equation}
The multiplicity of $[ab-3] \times [\alpha]$ for $\alpha\in \{(3),(1^3)\}$ is zero in the second summand (because $\mu\neq \nu$ implies that $ [(a^b)-\mu] [(a^b)-\nu] $ does not contain the trivial representation).  
 We now consider the first summand.
Let  $\alpha\in\{(3),(1^3)\}$, we can apply \cref{lem:SA-prodgroups} to obtain
\begin{align}\nonumber
\sum_{\mu\in\{(3),(1^3)\}}    \!\!\!\!
\langle X^2([(a^b)-\mu] \times   [\mu] )\mid [ab-3] \times [\alpha] \rangle
&= 
\sum_{\mu\in\{(3),(3^3)\}} \!\!\!\!
\langle X^2([(a^b)-\mu] ) \times [\mu]^2 \mid [ab-3] \times [\alpha] \rangle 
\\
\label{yasqueen}
&=
\begin{cases}
2 &\text{if $X=S$ and $\alpha=(3)$} \\
0 &\text{otherwise} 
\end{cases}\end{align}
 Thus it only remains to consider the contribution of $\mu=(2,1)$ to the $ [ab-3] \times [\alpha]$ isotypic summand. 
 We consider the $\alpha=(1^3)$ case as the $\alpha=(3)$ case can be argued in an identical fashion.  We have that 
 \begin{align}\begin{split}\label{splitme}
&\langle   X^2([(a^b)-(2,1)] \times   [2,1] )\mid [ab-3] \times [1^3] \rangle
\\   &\quad=  \frac{1}{6(n-3)!}  \sum_{(g,h)\in S_{ab-3}\times S_3} 	
  \!\!\!\!\!\!\!
     X^2([(a^b)-(2,1)] \times   [2,1] )(g,h) \times (-1)^{ \ell(h)}.
\end{split}
  \end{align}
We breakdown  the righthand-side of \cref{splitme} according to   the elements $h\in S_3$.  
We note that $(g,1)^2=(g,1)$ and $[2,1](1)=2$ and  therefore
\begin{align}\label{needt1}
 X^2([(a^b)-(2,1)] \times   [2,1] )(g,1)		&= X^2 (2[(a^b)-(2,1)](g)
\intertext{similarly,  $(g,(1,2))^2=(g,1)$ and $[2,1](1)=2$ and and $[2,1](1,2)=0$  and therefore}
\label{needt2}
 X^2([(a^b)-(2,1)] \times   [2,1] )(g,(1,2))		&= 
 \begin{cases}-	 [(a^b)-(2,1)](g^2) &\text{for }X=A \\
 	 [(a^b)-(2,1)](g^2) &\text{for }X=S 
 \end{cases}
\intertext{and finally,  $(g,(1,2,3))^2=(g,(1,3,2))$ and $[2,1](1,2,3)=-1=[2,1](1, 3,2)$ and   therefore}
\label{needt3}
 X^2([(a^b)-(2,1)] \times   [2,1] )(g,(1,2,3))		&= X^2( - [(a^b)-(2,1)])(g)
\end{align}
 We  can now 
decompose the righthand-side of \cref{splitme} for $X=A$ according to the conjugacy classes of $S_{3}$ (of size 1, 3, and 2 respectively) and substitute in \cref{needt1,needt2,needt3} and hence obtain 
$$
 \frac{1}{6(n-3)!}  \Bigg(\sum_{g\in S_{n-3}}    A^2( 2	[(a^b)-(2,1)] )(g) 	+  3	\sum_{g\in S_{n-3}}    	[(a^b)-(2,1)] (g^2)  		+ 2 	\sum_{g\in S_{n-3}}   S^2[(a^b)-(2,1)] (g) \Bigg)$$
 where the final term comes from the substitution $A^2 (- [\la])= S^2([\la])$.  
 We now consider the three terms in the above sum.
 We have that 
 \begin{align*}\textstyle
 \tfrac{1}{6(n-3)!}  \sum_{g\in S_{n-3}}    A^2( 2	[(a^b)-(2,1)] )(g) 
& = \tfrac{1}{6}\langle A^2( 2	[(a^b)-(2,1)] ) \mid [ab-3]\rangle 
\\
&=\tfrac{1}{6}( 2\langle    A^2[(a^b)-(2,1)]    \mid [ab-3]\rangle + 
\langle   
 [(a^b)-(2,1)] ^{2} \mid [ab-3]\rangle)
 \\
&=\tfrac{1}{6}
\intertext{where the final equality follows from \cref{prop:smalldepthconstituents}. Similarly,  we have that }
\textstyle
 \tfrac{ 3}{6(n-3)!}   	\sum_{g\in S_{n-3}}    	[(a^b)-(2,1)] (g^2)  
 &=
  \tfrac{ 1}{2 }    \langle 	[(a^b)-(2,1)] (g^2)   \mid [ab-3]\rangle 
  \\
& =
  \tfrac{ 1}{2 }    \langle 	S^2[(a^b)-(2,1)] -A^2[(a^b)-(2,1)]    \mid [ab-3]\rangle 
 \\
 &=   \tfrac{ 1}{2 } 
 \intertext{where the second  equality follows from \cref{ppppp2} and the third from \cref{prop:smalldepthconstituents}.
Finally, we have that }\textstyle
  \tfrac{ 2}{6(n-3)!}  	\sum_{g\in S_{n-3}}   S^2[(a^b)-(2,1)]  ^{(2)}
  &=
   \tfrac{ 1}{3 }    \langle 	S^2[(a^b)-(2,1)]     \mid [ab-3]\rangle \\
  &=
   \tfrac{ 1}{3 }       
\end{align*}
where again the second equality follows   from \cref{prop:smalldepthconstituents}.
Summing over these terms we obtain 
\begin{align}\label{noqueen1}
\langle   A^2([(a^b)-(2,1)] \times   [2,1] )\mid [ab-3] \times [1^3] \rangle=1.   
\intertext{In a similar fashion, one can show that }\label{noqueen2}
 \langle   S^2([(a^b)-(2,1)] \times   [2,1] )\mid [ab-3] \times [3] \rangle=1.
 \end{align}
 Thus, putting all of  \cref{yasqueen,noqueen1,noqueen2} into \cref{forlaters} we obtain 
 $$
 \langle 
 X^2(\la){{\downarrow}}_{S_{n-k}\times S_k} 
 \mid 
  [ab-3]\times [\alpha]\rangle = 
 \begin{cases}
 1 		&\text{if $X=A$ and $\alpha=(1^3)$}\\
 3 		&\text{if $X=S$ and $\alpha=(3)$}
 \end{cases}.
 $$
 Finally, the irreducible $S_n$-characters for which $[ab-3]\times [1^3]$ appears as a constituent in  the restriction are $[ab-3,1^3]$  and $[ab-2,1^2]$, the latter appears with coefficient 0 in $A^2([a^b] )$ and so the former must appear with coefficient 1, as required.  
The  irreducible $S_n$-characters for which $[ab-3]\times [3]$ appears as a constituent in  the restriction are $[ab-3,3]$,    $[ab-2, 2]$,  $[ab-1, 1]$ and  $[ab]$; the final three of which appear with coefficients  $1, 0, 1$ respectively  in $S^2([a^b] )$ and so the first must appear with coefficient $3-2=1$, as required.   
\end{proof}

\subsection{Squares of small depth characters}\label{somesmallprods}

We now turn to the squares of characters to
partitions of depth at most~2.
First we recall a result due to Malle and Magaard
on the squares of such characters
obtained in the context of determining
the situations when the
symmetric or alternating parts are irreducible \cite[Lemma 2]{MM}.

Below, for numbers $n,m\in \N$ the expression $\delta_{n \ge m}$
is defined to be~0 if $n<m$ and~1 if $n\ge m$.
The detailed description below makes
the monotonous behaviour of the
coefficients explicitly visible.
Note that the case of the character $[n-1,1]$
also follows from Theorem~\ref{thm:smalldepthconst},
and both $[n-1,1]$ and $[n-2,1^2]$ are also covered
as special cases of the more recent Theorem~\ref{thm:MW} \cite{MW}.

\begin{prop}\cite{MM}\label{prop:smalldepth}
Let $n\in \N$.
Then we have the following decompositions.
 For all $n\ge 3$:
\begin{align*}
S^2([n-1,1]) &=  [n]+[n-1,1]+ \delta_{n\ge 4} [n-2,2],
\qquad \quad
A^2([n-1,1])  =  [n-2,1^2].
\end{align*}
 For all $n\ge 4$:
\begin{align*}S^2([n-2,2]) &=  [n]+\delta_{n\ge 5}[n-1,1]
+ (1+\delta_{n \ge 6})[n-2,2]+\delta_{n\ge 7}[n-3,3]\\ &  \qquad
+\delta_{n\ge 5}[n-3,2,1]+\delta_{n\ge 8}[n-4,4]+\delta_{n\ge 6}[n-4,2^2],
\\ 
A^2([n-2,2]) &=  \delta_{n\ge 5}[n-2,1^2]
+\delta_{n\ge 6}[n-3,2,1]+[n-3,1^3]   +\delta_{n\ge 7}[n-4,3,1].
\end{align*} For all $n\ge 5$:
\begin{align*}
S^2([n-2,1^2]) &=  [n]+[n-1,1]+2[n-2,2]+
  +\delta_{n\ge 6}[n-3,3]+[n-3,2,1] +
  \delta_{n\ge 6}[n-4,2^2]+[n-4,1^4],
\\ 
A^2([n-2,1^2 ]) &=  [n-2,1^2]+[n-3,2,1]+[n-3,1^3]+\delta_{n\ge 6}[n-4,2,1^2].
\end{align*}\end{prop}


We now go one step further, to depth~3.
In \cite{MM}, the proof of Proposition~\ref{prop:smalldepthconstituents}
in the case of $[n-1,1]$ is given in detail,
while the proofs in the case of $[n-2,2]$ and $[n-2,1^2]$ are only hinted at;
as the proof of the result in case of $[n-3,3]$ is quite involved
we will give a detailed proof of this decomposition here.
The case of $[n-3,1^3]$ is
again covered by the recent
result on hooks \cite{MW}.
For $[n-3,2,1]$, the data suggest a formula
that is much more involved.

\begin{theorem}\label{thm:depth3}
Let $n\in \N$, $n\ge 6$.
Then we have the following decompositions.
\begin{align*}
S^2([n-3,3]) =
&[n]+\delta_{n\ge 7}[n-1,1]+ (1+\delta_{n\ge 8})[n-2,2]
+(\delta_{n\ge 7}+\delta_{n\ge 9})[n-3,3]
\\[4pt]
 &
+\delta_{n\ge 7}[n-3,2,1]
+(\delta_{n\ge 8}+\delta_{n\ge 10})[n-4,4]
+\delta_{n\ge 8}[n-4,3,1]
\\[4pt]
& 
+(1+\delta_{n\ge 8})[n-4,2^2]
+\delta_{n\ge 11}[n-5,5]
+\delta_{n\ge 9}[n-5,4,1]
\\[4pt]
&+\delta_{n\ge 9}[n-5,3,2]
+\delta_{n\ge 7}[n-5,2^2,1]
+\delta_{n\ge 12}[n-6,6]
+\delta_{n\ge 10}[n-6,4,2]
\\[7pt]
A^2([n-3,3]) = &
\delta_{n\ge 7}[n-2,1^2]+\delta_{n\ge 8}[n-3,2,1]+[n-3,1^3]
+(\delta_{n\ge 7}+\delta_{n\ge 9})[n-4,3,1]
\\[4pt]
&
+\delta_{n\ge 7}[n-4,2,1^2]+\delta_{n\ge 10}[n-5,4,1]
+\delta_{n\ge 8}[n-5,3,2]
\\[4pt]
&
+\delta_{n\ge 8}[n-5,3,1^2]
+\delta_{n\ge 11}[n-6,5,1]+\delta_{n\ge 9}[n-6,3^2]
\end{align*}
%

\end{theorem}

\begin{proof}
Set $\la=(n-3,3)$.
Since formulae for the product of 2-part partitions but also for
the multiplicity of constituents $[n-3,3]$ in arbitrary products
are available
(see \cite{BaOr, Rosas, Saxl, V}),
we know how $\cla^2$ decomposes into irreducibles.
The two expressions for the symmetric and alternating part
given in the assertion above sum to $\cla^2$,
and we have to show
that the distribution is indeed correct.
For this we will use induction; by computation (with Maple)
we know that the result holds up to $S_{12}$. 
So we can now assume that $n\ge 13$, and then all $\delta$-coefficients
appearing in the formula are~1.
 Using Lemma~\ref{lem:SA-sum} we obtain
\color{black}
\begin{align}\label{alabelCB}
S^2([n-3,3]){\downarrow}_{S_{n-1}}
= S^2([n-4,3]) + S^2([n-3,2]) + [n-4,3]  [n-3,2] .
\end{align}
By induction and Proposition~\ref{prop:smalldepthconstituents}, respectively,
we know how the first two summands decompose, and from the formula
for characters to 2-part partitions, we know the decomposition of the product
$[n-4,3]  [n-3,2]$ (e.g., from \cite{BaOr}):
 \begin{align}\label{alabelCB2}
\begin{split}[n-4,3]  [n-3,2] =  
&  [n-2,1]+[n-3,2]+[n-3,1^2]+2[n-4,3]+2[n-4,2,1]\\
&  +[n-5,4]+2[n-5,3,1]+[n-5,2^2]+[n-5,2,1^2]+[n-6,5]\\
&  +[n-6,4,1]+[n-6,3,2]
\end{split}
\end{align}
Altogether we obtain the following expression
for the restriction $S^2([n-3,3]){\downarrow}_{S_{n-1}}$:
\begin{align}\label{alabelCB3}
\begin{split} 
&2[n-1]+3[n-2,1]+[n-3,1^2]+5[n-3,2]+5[n-4,3]+4[n-5,4]\\
&+3[n-5,2^2]+4[n-4,2,1]+3[n-5,3,1]+2[n-6,5]+2[n-6,4,1]\\
&+2[n-6,3,2]+[n-6,2^2,1]+[n-7,6]+[n-7,4,2]+[n-3,1^2]+[n-5,2,1^2]
\end{split}
\end{align}
From Theorem~\ref{thm:smalldepthconst},
we already know
\[
S^2([n-3,3]) = [n] + [n-1,1] + 2[n-2,2] + \text{constituents of depth $>2$},
\]
and
\[A^2([n-3,3]) = [n-2,1^2] + \text{constituents of depth $>2$.}
\]
As mentioned, we also have the decomposition of the square $[n-3,3]^2$.
Next we want to apply Corollary~\ref{cor:samehooksums};  
it is easily checked that for even $n>6$, $[n-3,3]$ vanishes on elements
of cycle type $(\frac n2, \frac n2)$.
Hence we obtain that
$S^2([n-3,3])$ and $A^2([n-3,3])$ have the same number of constituents labelled by  hooks of
odd 
leg length; this implies that 
$[n-3,1^3]$
is (only) a constituent of $A^2([n-3,3])$.

From the first 3 terms in $S^2([n-3,3])$,
we get as contributions to the restriction $(\ref{alabelCB3})$:
\[
2[n-1]+3[n-2,1]+2[n-3,2].
\]
We now have to determine the three
constituents in $2[n-3,3]+2[n-3,2,1]$ (from $[n-3,3]^2$)
which belong to $S^2$
and contribute the other three constituents $[n-3,2]$ to $(\ref{alabelCB3})$.
Since $[n-3,1^2]$ appears only with multiplicity 1 in $(\ref{alabelCB3})$,
we deduce that $2[n-3,3]+[n-3,2,1]$ belongs to $S^2$, and
one constituent $[n-3,2,1]$ to $A^2$.
We also note that $[n-4,1^3]$ is not in the restriction $(\ref{alabelCB3})$
of $S^2$,
so $[n-4,2,1^2]$ must be in $A^2$.
Next, we need three constituents of $2[n-4,4]+3[n-4,3,1]$ in $S^2$
to get in total $5[n-4,3]$ in (\ref{alabelCB3}), and the other two will then
contribute the required $2[n-4,3]$ to the restriction of $A^2$.

By induction, in the restriction of $A^2$,
the term $[n-5,4]$ (and also  $[n-6,5]$)
only appears once, in the product, so
$A^2$ can contain at most one constituent of
$2[n-4,4]+[n-5,5]+2[n-5,4,1]$.
Also by induction, $[n-7,6]$ does not occur at all in the restriction of $A^2$, so $[n-6,6]$ only occurs in $S^2$, and then $[n-6,5,1]$ must occur (once)
in $A^2$; but then $[n-5,5]$ cannot also occur in $A^2$.
Furthermore, since $[n-7,4,2]$ does not occur in the restriction of $A^2$ by induction, the second constituent $[n-6,4,1]$ can only come from
$[n-5,4,1]$ in $A^2$, and hence $2[n-4,4]+[n-5,5]+[n-5,4,1]$ is in $S^2$.
To get a total of $5[n-4,3]$ in the restriction of $S^2$, we then have exactly one $[n-4,3,1]$ in $S^2$ (and $2[n-4,3,1]$ in $A^2$).
We note that $[n-4,2^2]$ is not in $A^2$, since this would give a surplus term $[n-4,2,1]$ in the restriction of $A^2$, so $2[n-4,2^2]$ is in $S^2$.
By induction, there is only one $[n-5,2^2]$ in the restriction of $A^2$, but
$3[n-5,2^2]$ in $S^2$. Hence $A^2$ can contain only one constituent of
$2[n-5,3,2]+[n-5,2^2,1]$.
But for the still missing $2[n-5,3,1]$ in the restriction of $A^2$,
2 constituents of $2[n-5,3,2]+[n-5,3,1^2]$ are needed in $A^2$.
Hence $[n-5,3,2]+[n-5,3,1^2]$ is in $A^2$, and $[n-5,3,2]+[n-5,2^2,1]$
in $S^2$.
Finally, from the restrictions we immediately conclude that $[n-6,4,2]$ belongs to $S^2$ and $[n-6,3^2]$ to $A^2$.

(A final check shows that the restrictions of the stated decompositions give indeed the right result.)
\end{proof}


\begin{rems}{\rm
(1)
While the patterns are not really clear from the cases of small $k$,
we will later see some general properties for the symmetric and
alternating part of $[n-k,k]^2$ (see Section~\ref{sec:2part}).

(2)
In the light of Proposition~\ref{prop:basicobs} and Remarks~\ref{rem:applicationtoproj}
it is worth pointing out again that all the splittings determined here in case of small depth partitions (and later for some other families of partitions)
immediately provide relations between the coefficients of projective characters.
Apart from the decomposition numbers, we also obtain relations
between some coefficients $\gen{[\rho_k]^2,[\mu]}$ appearing in the Saxl square.
%
}
\end{rems}


    \section{Splitting the square:
The sign constituent and its neighbour}\label{sec:sign-and-beyond}

We have already seen that for all $\la\in P(n)$,
the trivial character $[n]$ is a constituent of $S^2(\cla)$.
How about the sign character $[1^n]$?

Let $\la^t$ denote the transpose of $\la$.
It is well known that $[1^n]\cla = [\la^t]$, so
\[
\gen{\cla^2,[1^n]}= \gen{\cla,[\la^t]}=
\left\{
\begin{array}{cl}
1 & \text{if } \la=\la^t\\
0 & \text{otherwise}
\end{array}
\right.
\]
Thus, $[1^n]$ occurs as a constituent in the square $\cla^2$ only when
$\la$ is symmetric, and in this case with multiplicity~1.

The question for which symmetric
$\la$ the constituent $[1^n]$
occurs in $A^2(\cla)$ was answered
in \cite[Theorem 3.3]{GIP}; the proof
of this theorem required some
intricate results and tableaux combinatorics.
We state the answer below
and provide a new proof
involving the characters of the alternating groups
which seems more conceptual.

First, we recall some facts on the irreducible characters
of the alternating group~$A_n$.
For $\la\ne \la^t \in P(n)$,
the restriction $\cla{\downarrow}_{A_n}=\{\la\}$
is irreducible, while for $\la=\la^t$, the
restriction is a sum of two different but algebraically conjugate
irreducible characters $\{\la\}_+$ and $\{\la\}_-$.
In this way we obtain all irreducible characters of~$A_n$.

In particular, the characters $\{\la\}_+$ and $\{\la\}_-$ have the same Frobenius--Schur indicator (mentioned in Section~\ref{sec:prelim}), and when they are non-real, they are complex conjugate to each other.
In fact, for $\la=\la^t$ with principal hook lengths $h_1, \ldots,h_d$ (for $d=dl(\la)$ the Durfee length of $\la$),  
the only possibly non-integral values of $\{\la\}_{\pm}$ are
$$\textstyle\frac 12 (\epps_\la \pm \sqrt{\epps_\la \prod_{j=1}^d h_j}),$$
where $\epps_\la = (-1)^{(n-d)/2}$;
in particular,
the characters $\{\la\}_{\pm}$ are real
exactly if $\epps_\la = 1$.

\begin{theorem}\label{thm:GIP}
Let $\la\in P(n)$ be a symmetric partition; set $d=dl(\la)$.
Then $[1^n]$ is a constituent of
$S^2(\cla)$ exactly if
the characters $\{\la\}_{\pm}$ are real, i.e., if $n\equiv d \mod 4$.
Equivalently,
\[
sg(\la,(1^n)) - ag(\la,(1^n)) = (-1)^{\frac{n-d}2} .
\]
\end{theorem}

\begin{proof}
Since $\cla{\downarrow}_{A_n}=\{\la\}_+ +\{\la\}_-$,
  using Lemma~\ref{lem:SA-sum}
\color{black}
we have for $X\in \{S,A\}$:
\[
X^2(\cla){\downarrow}_{A_n} =
X^2(\{\la\}_+) + X^2(\{\la\}_-) + \{\la\}_+  \{\la\}_-.
\]
When the characters $\{\la\}_{\pm}$ are non-real,
we have
\[
1=\gen{\{\la\}_+,\{\la\}_+}=\gen{\{\la\}_+  \{\la\}_-,\{n\}},
\]
while in the real case $\gen{\{\la\}_+  \{\la\}_-,\{n\}}=0$.
Now the trivial character  $\{n\}$ only comes from the restriction
of $[n]$ and $[1^n]$ to $A_n$, and we already know that both
occur only once in $\cla^2$, and that
$[n]$ is a constituent of $S^2(\la)$.
Hence, when $\{\la\}_{\pm}$ are non-real,
$[1^n]$ must occur in $A^2(\cla)$.
Now assume that the characters $\{\la\}_{\pm}$ are real;
by the reasoning above,
at least one of $S^2(\{\la\}_{\pm})$ has to contain $\{n\}$,
but since $\{\la\}_{\pm}$ have the same Frobenius--Schur indicator,
then both $S^2(\{\la\}_{\pm})$ contain $\{n\}$.
Thus, in the real case, $S^2(\cla)$ has to contain both, $[n]$ and $[1^n]$.

Finally, it was already pointed out above
that the characters $\{\la\}_{\pm}$ are non-real
exactly if $\epps_\la = (-1)^{(n-d)/2} = -1$, so we are done.
\end{proof}


For the location of the sign character in one of the two parts
of the Saxl square we then have the following
consequence.

\begin{cor}
Let $k\in \N$, $k>1$, $n=k(k+1)/2$. 
Then the following holds.
\begin{enumerate}
\item[{(1)}]
If $k \not\equiv 2 \mod 4$,
then
$sg(\rho_k,(1^n))=1$, $ag(\rho_k,(1^n))=0$.
\item[{(2)}]
If $k\equiv 2 \mod 4$, then
$sg(\rho_k,(1^n))=0$, $ag(\rho_k,(1^n))=1$.
\end{enumerate}
\end{cor}
\begin{proof}
The diagonal length of $\rho_k$ is $d=\lfloor \frac{k+1}2 \rfloor$.
Thus, when $k$ is odd, we have
\[
\frac{n-d}2 = \tfrac 14 (k(k+1)-(k+1)) = \tfrac 14 (k-1)(k+1),
\]
which is always even.  
When $k$ is even, we have
\[
\frac{n-d}2 = \tfrac 14 (k(k+1)-k) = \tfrac 14 k^2,
\]
which is even exactly when $k\equiv 0 \mod 4$.
\end{proof}


We want to build on Theorem~\ref{thm:GIP} and determine next the distribution of
$[2,1^{n-2}]$ in the Kronecker square $\cla^2$.

\begin{theorem}\label{thm:lowendconst}
Let $\la\in P(n)$, with~$d=dl(\la)$ and $r_1$ the number of
removable boxes.
\begin{enumerate}
\item
Assume $|\la\cap \la^t|<n-1$. Then
\[
sg(\la,(2,1^{n-2}))=0=ag(\la,(2,1^{n-2})).
\]

\item
Assume $|\la\cap \la^t|=n-1$. Then
\[
\begin{array}{rl}
sg(\la,(2,1^{n-2}))=1, ag(\la,(2,1^{n-2}))=0
&
\text{if } \frac{n-1-d}2 \text{ is even,}\\[7pt]
sg(\la,(2,1^{n-2}))=0, ag(\la,(2,1^{n-2}))=1
&
\text{if } \frac{n-1-d}2 \text{ is odd.}
\end{array}
\]

\item
Assume $|\la\cap \la^t|=n$, i.e., $\la=\la^t$.
Set $\epps_\la=(-1)^{(n-d)/2}$.
Then
\[
\begin{array}{rl}
sg(\la,(2,1^{n-2}))
= \frac{r_1-1}2 =
ag(\la,(2,1^{n-2}))
&\text{if  $r_1$ is odd}\\[7pt]
sg(\la,(2,1^{n-2}))
= \lfloor \frac{r_1-\epps_\la}2 \rfloor,
ag(\la,(2,1^{n-2}))
= \lfloor \frac{r_1+\epps_\la}2 \rfloor
&\text{if  $r_1$ is even.}
\end{array}
\]

\end{enumerate}
\end{theorem}

\begin{proof}
First of all, we note that $|\la\cap \la^t|$ is the maximal length
of a constituent in $\cla^2$ (see \cite{Dvir}).
Hence $[2,1^{n-2}]$ can only occur as a constituent in $\cla^2$
if $|\la\cap \la^t| \ge n-1$. This gives~(1).

Furthermore, when $|\la\cap \la^t|=n-1$, then we have (using \cite{Dvir}):
\[
\gen{\cla^2,[2,1^{n-2}]} = \gen{[1]  [1],[1]} =1.
\]
So for (2), we only have to determine whether $[2,1^{n-2}]$ belongs
to $S^2(\cla)$ or $A^2(\cla)$. We keep in mind that $\cla^2$ does
not contain $[1^n]$.
\\
We now consider the restriction of $\cla$ to $S_{n-1}$.
Let again $X\in\{A,S\}$. We have already seen before that
we obtain
\begin{equation}\label{(*)}
 X^2(\cla){\downarrow}_{S_{n-1}} =
\sum_C X^2([\la_C])  + \sum_{C\ne D} [\la_C]  [\la_D],
\end{equation}
where $C,D$ run over the removable boxes of $\la$.
Let $B$ be the box in $\la/(\la\cap \la^t)$.
Now, $[1^{n-1}]$ is a constituent in a product
$[\la_C]  [\la_D]$ (appearing in the sum above)
if and only if $\la_C=\la_D^t$. But if $C=B$,
this cannot hold for any $D\ne C$;
then, $C\ne D$ must be in transpose positions, but then
$\la_C \ne \la_D^t$, by comparing $B$ and its transpose.
So $[1^{n-1}]$ can occur as a constituent
in $X^2(\cla){\downarrow}_{S_{n-1}}$, if and only if
it is in $X^2([\la_B])$.
For $X=S$ ($X=A$, respectively), we know by Theorem~\ref{thm:GIP},
that this is the case
if and only if $\frac12({n-1-dl(\la_C)})$ is even (odd, respectively).
As $C$ is not on the diagonal of $\la$,
 we have that $$\tfrac12({n-1-dl(\la_C)}) = \tfrac12({n-1-d}).$$
Since $[1^n]$ does not appear in $\cla^2$,
$[1^{n-1}]$ can occur as a constituent
in $X^2(\cla){\downarrow}_{S_{n-1}}$,
if and only if $[2,1^{n-2}]$ is a constituent
of $X^2(\cla)$. Hence $[2,1^{n-2}]$ is a constituent
of $S^2(\cla)$ ($A^2(\cla)$, respectively) if and only if
$\frac12({n-1-d})$ is even (odd, respectively), and thus
(2) is proved.

Now we turn to (3), i.e., now $\la=\la^t$. In this case, we have
(see Proposition \ref{prop:smalldepthconstituents})
\[
g(\la,\la,(2,1^{n-2}))=g(\la,\la,(n-1,1))=r_1-1.
\]
As in the previous case, we consider the restriction of
$\cla$ to $S_{n-1}$ and look at \cref{(*)}.

First we consider the case that $r_1$ is odd, i.e., there is a
removable box, $B$ say,  on the diagonal of $\la$.
This time, a product $[\la_C]  [\la_D]$ (appearing in \cref{(*)} above)
can contain a constituent $[1^{n-1}]$, namely, exactly when
$C,D$ are a pair of transpose boxes in $\la$.
As the box $B$ is never involved,
we obtain exactly $\frac{r_1-1}2$ such constituents $[1^{n-1}]$.
In the first sum on the right side of \cref{(*)}, we can have at most
one constituent $[1^{n-1}]$, namely in $X^2([\la_B])$,
and we know it is in $S^2([\la_B])$ ($A^2([\la_B])$, respectively),
when $\frac12({n-1-(d-1)})=\frac12({n-d})$ is even (odd, respectively).
We thus have found $r_1$ constituents $[1^{n-1}]$
in $\cla^2{\downarrow}_{S_{n-1}}$, which can only come from
the part $[1^n]+(r_1-1)[2,1^{n-2}]$ in $\cla^2$;
we also  know from the previous arguments that
$[1^n]$ is in $X^2(\cla)$ if and only if $[1^{n-1}]$
is in $X^2([\la_B])$.
Thus we must have exactly $\frac12({r_1-1})$ constituents
$[2,1^{n-2}]$ in both $S^2(\cla)$ and $A^2(\cla)$, as claimed.

Next we consider the case that $r_1$ is even, i.e., there is no
removable box on the diagonal of $\la$.
Similar as in the previous case, we find exactly $\frac{r_1}2$ products
$[\la_C]  [\la_D]$ with a constituent $[1^{n-1}]$, to pairs of transposed  boxes in $\la$;  lying above these, we find $\frac{r_1}2$ constituents in
$X^2(\cla)$.
In this case no $\la_C$ is symmetric, so none of $X^2([\la_C])$ can contain a constituent $[1^{n-1}]$.
So the one constituent $[1^n]$ in either $S^2(\cla)$ or $A^2(\cla)$ (depending on $\epps_\la$ being 1 or $-1$, respectively), restricts to one of the
constituents $[1^{n-1}]$ in $\sum_{C\ne D} [\la_C]  [\la_D]$,
while all the other such constituents have to come from
$\frac{r_1}2-1$ constituents $[2,1^{n-2}]$ in $S^2(\cla)$ or $A^2(\cla)$,
respectively, and the other part of the Kronecker square has
$\frac{r_1}2$ constituents $[2,1^{n-2}]$.
It is easily seen that this is exactly the assertion in (3) in the case of even $r_1$.
\end{proof}

   \section{Splitting the square: 2-part partitions}\label{sec:2part}

For the Kronecker square of a 2-part partition $\la$ even the
coefficients $g(\la,\la,\mu)$ are difficult to determine
(see the earlier comments on Remmel's formula),
and we have also seen that even in case of partitions of small depth
the formulae are intricate.
So we will focus here on some particularly nice squares.
We first state the decomposition of some
multiplicity-free products of characters to 2-part partitions.
 These follow from the general formulae,
and were known already as part of the easy direction of the
classification conjecture on multiplicity-free products (see \cite{BeBo});
they have also appeared explicitly in some more recent papers.
The fact that only partitions of length at most~4 can appear
is easy to see as a consequence of \cite{Dvir}.  
 
\begin{prop}\label{prop:2partmfKronecker}
 For $k\in \N_0$ and $n=2k+1$, we have that 
\begin{equation}\tag{1}
\-
[k+1,k]^2 = \sum_{\la\in P_4(n)} [\lambda]\:.
\end{equation}
For $k\in \N$ and $n=2k$, we let  $E_4(n)$ and $O_4(n)$ denote
the sets of partitions $\la\in P_4(n)$ into only even parts and only odd parts, respectively. 
Then
\begin{equation}\tag{2}
\-
[k,k]^2 =\sum_{\lambda \in E_4(n)} [\lambda] +
\sum_{\lambda \in O_4(n)}[\lambda]\:.
\end{equation}
 Let $k\in \N$, $n=2k$.  Let  $E_2O_2(n)$ denote
the set of partitions $\la\in P_4(n)$ with two even and two odd parts
(again, 0 is considered an even part).
Then
\begin{equation}\tag{3}
\-
[k,k]  [k+1,k-1] =\sum_{\lambda \in E_2O_2(n)} [\lambda]
\:.
\end{equation}
 \end{prop}

Notably the square in (2) above was also computed by Manivel \cite{Man10},
who added \emph{``that it would be interesting to understand the splitting of
$[n, n]   [n, n]$ into its symmetric and skew-symmetric parts"}.
We will answer his question and provide some
further square splittings below.

\color{black}
 We further note that (1) can easily be deduced from (2) by restriction.
\color{black}

First we show that all characters to $E_4$-partitions occur in the symmetric part.
\begin{prop}\label{part1}
For $k\in \N$ and $n=2k$, we have
\[
sg((k,k),\la)=1 \quad \text{for all } \lambda \in E_4(n) .
\]
\end{prop}

\begin{proof}
For $j\in \{1,2,3\}$ we already know and for $j=4$ we check by computer that
$sg((j,j),(2^j))=1$.
Now if $\la\in E_4(n)$, we can decompose $\la$ into a sum of double-columns,
i.e., $\la = \sum_{j=1}^{\la_1/2} \la^{(j)}$,
where $\la^{(j)} \in \{(2),(2^2),(2^3),(2^4)\}$, for all~$j$.
Hence by Proposition~\ref{prop:2partmfKronecker} and the semigroup property stated in Proposition~\ref{prop:Ressayre}, we obtain
$sg((k,k),\la)=1$.
\end{proof}

Next we prove with considerable more work that for all $\la \in O_4(2k)$
we have $sg((k,k),\la)=0$.   This will then give the desired splitting.

\color{black} In this section we deal with characters and modules induced from Young subgroups and wreath product subgroups,  the irreducible characters and modules, and the symmetric powers of all these modules.  
It is thus easier to use module notation (rather than introduce new character notation) throughout this section (in contrast to earlier sections).
\color{black}
  We first recall a couple of results on decomposing induced representations and 
 Kronecker products, which will be essential in what follows.  
 \begin{prop}[\cite{thrall}]
For $k\geq 1$, we have that 
$${\rm ind}_{S_k \wr S_2  }^{S_{2k}}(\mathbb{C})= \bigoplus_{\mu \in P_2(k)}S^{ 2\mu}_\mathbb{C}$$
\end{prop}

\begin{prop}[\cite{BWZ10}]\label{(dd)}
For $ 1\leq k \leq d$ we have that 
\begin{align}\label{formula}
  S^{(d,d)}_\mathbb{C}  \otimes 
    S^{(d+k,d-k) }_\mathbb{C}
    =
    \bigoplus _{i=0}^k S^{(k+i,k,i)P}_\mathbb{C}
    \bigoplus _{i=1}^k S^{(k+1+i,k+1,i)P}_\mathbb{C}
    \end{align}
where  $\gamma P=\{ \mu \mid \mu-\gamma\in O_4(2d)  \text {   or  } \mu-\gamma\in E_4(2d)\}.$ 
\end{prop}

We are now ready to begin proving  that no 
 $sg((k,k),\la)=0$ for all $\la \in O_4(2k)$. 
  The first step toward this goal is to understand the symmetric powers of Young permutation modules and their quotients.

\begin{prop}\label{youngrul}
For $\la\in P_2(n)$, we have that 
\begin{align}\label{decompY}
S^2(M^\la_\mathbb{C}) = \bigoplus_{0\leq k \leq \la_2} \bigoplus_{\alpha \in P_2(k)  }
{\rm ind}_{S_{n-2k}\times S_{2k }}^{S_n}  (M^{(\la_1-k,\la_2-k)}_\mathbb{C} \boxtimes  S^{2\alpha}_\mathbb{C} )
\end{align}
and hence the multiplicities  $[S^2(M^\la_\mathbb{C}): S^\mu_\mathbb{C}]$ can be calculated using Young's rule.
\end{prop}

\begin{proof}
We consider the action of $S_n$ on the   basis
 $
\{	\sts \otimes \stt+\stt \otimes \sts \mid \sts,\stt \in {\sf RStd}(\la)\}.
 $
We define the {\sf intersection number} of a pair of row-standard tableaux as follows 
   $$
  \imath(\sts,\stt) = \la_2- 
   |\{\sts(2,c)\mid 1\leq c \leq \la_2\} \cap \{\stt(2,c)\mid 1\leq c \leq \la_2\}  | $$
and we notice that  
    $\imath(\sts, \stt)= \imath(g(\sts), g(\stt))$ for $g\in S_n$ 
   and so 
   $S^2(M(\la)=\oplus_{0\leq k \leq \la_2}	M_k(\la)$ 
   where 
   $$M_k(\la)= \{\sts \otimes \stt+\stt \otimes \sts \mid 
   \sts,\stt \in {\sf RStd}(\la), 	\imath(\sts, \stt)= k		\}.$$
Each    $M_k(\la)$ is transitive permutation module of $S_{2n}$ with stabiliser given by
$$S_{\la_1-k ,\la_2-k} \times S_{k}\wr S_{2}$$
and so the result follows.  
\end{proof}

\begin{cor}\label{decompY2}
For $d\geq 1$, we have that $S^2(M^{(d+1,d-1}_\mathbb{C}) $ is a submodule of $S^2(M^{(d,d)}_\mathbb{C}) $ and the quotient module is as follows 
\begin{align*} 
S^2(M^{(d,d)}_\mathbb{C}) /  S^2(M^{(d+1,d-1)}_\mathbb{C})  
\cong  
\bigoplus_{0\leq k \leq d
} \bigoplus_{\alpha \in P_2(k)  }
 {\rm ind}_{S_{n-2k}\times S_{2k }}^{S_n}  (S^{(d-k,d-k)}_\mathbb{C} \boxtimes   S^{2\alpha}_\mathbb{C} ).
\end{align*}
\end{cor}
\begin{proof}
Setting   $\la=(d,d)$,  the righthand-side of \ref{decompY} has $d$ terms (the $M_k(d,d)$ for $0\leq k \leq d$).
Setting   $\la=(d+1,d-1)$,  the righthand-side of \ref{decompY} has  $(d-1)$ terms
(the $M_k(d+1,d-1)$ for $0\leq k \leq d-1$).  
For each $0\leq k \leq d-1$ we have that  $M_k(d+1,d-1)$ is a submodule of $M_k(d,d)$ and the quotient module is isomorphic to 
$$\bigoplus_{\alpha \in P_2(k)  }
 {\rm ind}_{S_{n-2k}\times S_{2k }}^{S_n}  (S^{(d-k,d-k)}_\mathbb{C} \boxtimes   S^{2\alpha}_\mathbb{C} ), 
$$
and the result follows.
\end{proof}
 
   \begin{prop}\label{part2}
 For $\mu \in O_4(n)$, we have that 
 $[S^2(S^{(d,d)}_\mathbb{C}): S^\mu_\mathbb{C}] =0$.

\end{prop}
\begin{proof}  \color{black}
For $\la,\nu \in P_2(n)$, we have that  $|{\rm SStd}(\la,\nu)|=1$ if $\la \trianglerighteq \nu$ (and is 0 otherwise).  Therefore 
 $$   
[M^{(d,d)}_\mathbb{C} ]
=
[S^{(d,d)}_\mathbb{C} ] + [M^{(d+1,d-1)}_\mathbb{C} ]$$
by Young's rule.  \color{black}
Rearranging this,  we note that  \begin{align*}
\chi_S^2 [S^{(d,d)}_\mathbb{C} ] 
&=
\chi^2_S([M^{(d,d)}_\mathbb{C} ]-[M^{(d+1,d-1)}_\mathbb{C} ])
\\
&=\chi^2_S [M^{(d,d)} _\mathbb{C}]
-
\chi^2_S [M^{(d+1,d-1)}_\mathbb{C} ]
-
[S^{(d,d)}_\mathbb{C}] \otimes  [M^{(d+1,d-1)}_\mathbb{C} ].
\end{align*}
 We need to verify,  for  
 $\mu\in O_4(n)$,  that the following holds 
\begin{equation}\label{alabelforaneqn}
\langle \chi^2_S [M^{(d,d)} _\mathbb{C}]
-
\chi^2_S [M^{(d+1,d-1)}_\mathbb{C} ] \mid [S^\mu_\mathbb{C}]\rangle = 
 \langle 
[S^{(d,d)}_\mathbb{C}] \otimes  [M^{(d+1,d-1)}_\mathbb{C} ]\mid [S^\mu_\mathbb{C}]\rangle  .
\end{equation}
We calculate the lefthand-side of equation (\ref{alabelforaneqn}) using Proposition \ref{decompY2}  and the righthand-side of equation (\ref{alabelforaneqn}) using Proposition \ref{(dd)}.  

 
\bigskip
\noindent{\bf Calculating the lefthand-side of equation  \ref{alabelforaneqn}. }
For $\mu \in O_4(n)$, we will  calculate the multiplicity of any $S^\mu_\mathbb{C}$ in the righthand-side of \ref{decompY2} using the Littlewood--Richardson  rule.  
Namely, we will calculate 
$$
\sum_{0\leq k \leq d} 
\sum_{\alpha\in P_2(k)} |{\sf LR}(\mu \setminus (d-k,d-k) , 2\alpha)|
$$
for $\mu \in O_4(2d)$.   Now, since $ (d-k,d-k)$ and $\alpha$ are both two-row partitions and   $\mu$ is a four-row partition, 
we deduce that  necessary conditions for $ |{\sf LR}(\mu \setminus (d-k,d-k) , 2\alpha)|\neq 0$ are  that  $d-k\geq \mu_3$ and 
 $d\geq \mu_2\geq d-k$.  Therefore we can set $m=d-k$ and  rewrite the above sum as follows 
 $$
\sum_{\mu_3	\leq m \leq \mu_2} 
\sum_{\alpha\in P_2(d-m)} |{\sf LR}(\mu \setminus (m,m) , 2\alpha)|.
$$
For any ${\sf T} \in  {\sf LR}(\mu \setminus (m,m) , 2\alpha)$, we have that 
  $${\sf T}(3,c)=1 \qquad  {\sf T}(4,c)=2 $$
 for all $1\leq c \leq \mu_4$.  Furthermore, 
 $${\sf T}(1,c)=1\qquad {\sf T}(2,c)=2$$
 for all $\mu_2\leq c\leq \mu_1$. 
 At this point,  we split into  two families of tableaux, depending on the parity of 
$\mu_2-m$. 
 We will use without further mention the fact 
  that $\mu_1-\mu_2, \mu_2-\mu_3$, $2\alpha_1$ and $2\alpha_2$ are even numbers and that 
  $\mu_3-\mu_4$ is odd.  

 If $ \mu_2-m$ is even, then we   fill the remaining  entries
 ${\sf T}(3,c)$ with an odd number of $2$s (and an odd number of $1$s).  We note that the total number of $2$s used in filling these boxes must be less than or equal to $\mu_1-\mu_2$, in order to satisfy the semistandard condition.    Thus the total number of such tableaux is given by
\begin{align}\label{sum1}
 \tfrac{1}{2}\min\{\mu_3-\mu_4,\mu_1-\mu_2\} \times  \tfrac{1}{2}(\mu_2-\mu_3+2)
\end{align}
 where there are $\tfrac{1}{2}\min\{\mu_3-\mu_4,\mu_1-\mu_2\} $ choices for the number of $2s$ in the third row and  where the
  value $\mu_3\leq  m\leq \mu_2$ must satisfy the parity condition.  

 If $m-\mu_2$ is odd, then we   fill the remaining entries
 ${\sf T}(3,c)$ with an even number of $2$s (and an even number of $1$s).  
  Similarly to the above, the total number of such tableaux is given by
\begin{align}\label{sum2}
 \tfrac{1}{2}(  \min\{\mu_3-\mu_4,\mu_1-\mu_2\}+2) \times  \tfrac{1}{2}(\mu_2-\mu_3 ).
\end{align}
 Therefore, summing over (\ref{sum1})  and  (\ref{sum2}) we deduce that   the lefthand-side of equation \ref{alabelforaneqn} is equal to 
\begin{align*}    \tfrac{1}{2}
 (\min\{\mu_3-\mu_4,\mu_1-\mu_2\} (\mu_2-\mu_3)
 +
 \min\{\mu_3-\mu_4,\mu_1-\mu_2\}+ (\mu_2-\mu_3))
\end{align*} for $\mu \in O_4(n)$.

 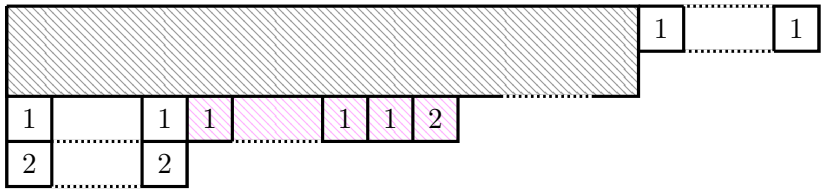
\begin{figure}[ht!]
 $$
 \begin{tikzpicture}[scale=0.6]
   
  \fill[pattern=north west lines,pattern color=magenta!30](4,-3)rectangle (10,-2);

  \draw[very thick,pattern=north west lines,pattern color=gray!80](0,0)--++(-90:2)--++(0:14)--++(90:2)  --(0,0);

  \path(0,0)--++(-90:2)--++(0:14)--++(90:2) coordinate (la1)--(0,0);

 \draw[very thick](la1)--++(0:1) coordinate (dot3)--++(0:2)coordinate (dot4)--++(0:1)--++(-90:1)--++(180:1)
 coordinate (dot5)--++(180:2)
 coordinate (dot6) --++(180:1);

 \draw[very thick](0,0)--++(-90:4*1)
 --++(0:1*1) coordinate (dot1) 
 --++(0:2*1) coordinate (dot2)  
  --++(0:1*1)  
    --++(90:1*1)  
      --++(0:1*1) 
      coordinate (dot7)
            --++(0:2*1) 
      coordinate (dot8)
            --++(0:3*1) 
           --++(90:1*1)
 ;

 \draw[very thick](dot7)--++(180:2) 
       coordinate (dot9)--++(180:2) 
       coordinate (dot10)--++(180:1) 
;

 \path 
 (0,0)--++(-90:4*1)
 --++(0:1*1)  
 --++(0:2*1) 
  --++(0:1*1)  
    --++(90:1*1)  
      --++(0:6*1)      --++(90:1*1)
      --++(0:1) coordinate (X)       --++(0:2) coordinate (Y)         --++(0:1)  --++(90:1)--++(0:4)--++(90:1);

  \draw[line width=1.3,white ,densely dotted] 
 (dot1)--(dot2)
  (dot3)--(dot4)
   (dot5)--(dot6)
      (dot7)--(dot8)
            (dot9)--(dot10) (X)--(Y);

 \clip 
 (0,0)--++(-90:4*1)
 --++(0:1*1)  
 --++(0:2*1) 
  --++(0:1*1)  
    --++(90:1*1)  
      --++(0:6*1)      --++(90:1*1)
      --++(0:1) coordinate (X)       --++(0:2) coordinate (Y)         --++(0:1)  --++(90:1)--++(0:4)--++(90:1);

     \foreach \i in {0,1,3,4,5,7,8,9}
         {	\draw[very thick](\i,-2)--++(-90:4);	}
        \foreach \i in          { 14,15,17,18}
  {	\draw[very thick](\i,0)--++(-90:4);	}
   \foreach \i in {0,1,...,18}
 {	\path(\i+0.5,0)--++(-90:0.5) coordinate (a\i);
 	\path(\i+0.5,0)--++(-90:1.5) coordinate (b\i);
	\path(\i+0.5,0)--++(-90:2.5) coordinate (c\i);
	\path(\i+0.5,0)--++(-90:3.5) coordinate (d\i);	}


 \draw(a14) node {$1$};
  \draw(a17) node {$1$};
    \draw(c0) node {$1$};
      \draw(c3) node {$1$};
        \draw(c4) node {$1$};
    \draw(d0) node {$2$};
        \draw(d3) node {$2$};
        
        \draw(c7) node {$1$};        
        \draw(c8) node {$1$};        
       \draw(c9) node {$2$};                
 
 \end{tikzpicture}
 $$

 \caption{A generic example  of a semistandard tableau  for $\mu_2-m=0 $   
  and $\alpha$   maximal    in the dominance ordering.
We note that  there are an odd number of 2s (and therefore an  odd number of 1s) in the pink region.  
   We note that for a fixed value $m$, the \emph{ only} choices 
 we can make are in the pink region.}
\end{figure}

\begin{figure}[ht!]
$$
  \begin{tikzpicture}[scale=0.6]
   
  \fill[pattern=north west lines,pattern color=magenta!30](4,-3)rectangle (10,-2);

 \draw[pattern=north west lines,pattern color=gray!80](0,0)--++(-90:2)--++(0:13)--++(90:2)  --(0,0);
 
  \draw[very thick](0,0)--++(-90:2)--++(0:14)--++(90:2) coordinate (la1)--(0,0);

  \draw[very thick] (la1)--++(-90:1)--++(180:1);

 \draw[very thick](la1)--++(0:1) coordinate (dot3)--++(0:2)coordinate (dot4)--++(0:1)--++(-90:1)--++(180:1)
 coordinate (dot5)--++(180:2)
 coordinate (dot6) --++(180:1);

 \draw[very thick](0,0)--++(-90:4*1)
 --++(0:1*1) coordinate (dot1) 
 --++(0:2*1) coordinate (dot2)  
  --++(0:1*1)  
    --++(90:1*1)  
      --++(0:1*1) 
      coordinate (dot7)
            --++(0:2*1) 
      coordinate (dot8)
            --++(0:3*1) 
           --++(90:1*1)
 ;

 \draw[very thick](dot7)--++(180:2) 
       coordinate (dot9)--++(180:2) 
       coordinate (dot10)--++(180:1) 
;
       
 \draw[line width=1.3,white ,densely dotted] 
 (dot1)--(dot2)
  (dot3)--(dot4)
   (dot5)--(dot6)
      (dot7)--(dot8)
            (dot9)--(dot10);
 
 \path 
 (0,0)--++(-90:4*1)
 --++(0:1*1)  
 --++(0:2*1) 
  --++(0:1*1)  
    --++(90:1*1)  
      --++(0:6*1)      --++(90:1*1)
      --++(0:1) coordinate (X)       --++(0:1) coordinate (Y)         --++(0:2)  --++(90:1)--++(0:4)--++(90:1);

 \draw[line width=1.3,white ,densely dotted] 
  (X)--(Y);

 \clip 
 (0,0)--++(-90:4*1)
 --++(0:1*1)  
 --++(0:2*1) 
  --++(0:1*1)  
    --++(90:1*1)  
      --++(0:6*1)      --++(90:1*1)
      --++(0:1) coordinate (X)       --++(0:1) coordinate (Y)         --++(0:2)  --++(90:1)--++(0:4)--++(90:1);

        \foreach \i in {0,1,3,4,5,7,8,9}
         {	\draw[very thick](\i,-2)--++(-90:4);	}
        \foreach \i in          {13,14,15,17,18}
  {	\draw[very thick](\i,0)--++(-90:4);	}

     \foreach \i in {0,1,3,4,5,7,8,9}
         {	\draw[very thick](\i,-2)--++(-90:4);	}
        \foreach \i in          {13,14,15,17,18}
  {	\draw[very thick](\i,0)--++(-90:4);	}
   \foreach \i in {0,1,...,18}
 {	\path(\i+0.5,0)--++(-90:0.5) coordinate (a\i);
 	\path(\i+0.5,0)--++(-90:1.5) coordinate (b\i);
	\path(\i+0.5,0)--++(-90:2.5) coordinate (c\i);
	\path(\i+0.5,0)--++(-90:3.5) coordinate (d\i);	}


 \draw(a13) node {$1$};
 \draw(b13) node {$2$};

 \draw(a14) node {$1$};
  \draw(a17) node {$1$};
    \draw(c0) node {$1$};
      \draw(c3) node {$1$};
        \draw(c4) node {$1$};
    \draw(d0) node {$2$};
        \draw(d3) node {$2$};
        
        \draw(c7) node {$1$};        
        \draw(c8) node {$1$};        
       \draw(c9) node {$1$};                
 
 \end{tikzpicture}$$

 $$
  \begin{tikzpicture}[scale=0.6]
   
  \fill[pattern=north west lines,pattern color=magenta!30](4,-3)rectangle (10,-2);

   \draw[pattern=north west lines,pattern color=gray!80](0,0)--++(-90:2)--++(0:13)--++(90:2)  --(0,0);
 
  \draw[very thick](0,0)--++(-90:2)--++(0:14)--++(90:2) coordinate (la1)--(0,0);

  \draw[very thick] (la1)--++(-90:1)--++(180:1);

 \draw[very thick](la1)--++(0:1) coordinate (dot3)--++(0:2)coordinate (dot4)--++(0:1)--++(-90:1)--++(180:1)
 coordinate (dot5)--++(180:2)
 coordinate (dot6) --++(180:1);

 \draw[very thick](0,0)--++(-90:4*1)
 --++(0:1*1) coordinate (dot1) 
 --++(0:2*1) coordinate (dot2)  
  --++(0:1*1)  
    --++(90:1*1)  
      --++(0:1*1) 
      coordinate (dot7)
            --++(0:2*1) 
      coordinate (dot8)
            --++(0:3*1) 
           --++(90:1*1)
 ;

 \draw[very thick](dot7)--++(180:2) 
       coordinate (dot9)--++(180:2) 
       coordinate (dot10)--++(180:1) 
;
       
 \draw[line width=1.3,white ,densely dotted] 
 (dot1)--(dot2)
  (dot3)--(dot4)
   (dot5)--(dot6)
      (dot7)--(dot8)
            (dot9)--(dot10);
 
 \path 
 (0,0)--++(-90:4*1)
 --++(0:1*1)  
 --++(0:2*1) 
  --++(0:1*1)  
    --++(90:1*1)  
      --++(0:6*1)      --++(90:1*1)
      --++(0:1) coordinate (X)       --++(0:1) coordinate (Y)         --++(0:2)  --++(90:1)--++(0:4)--++(90:1);

 \draw[line width=1.3,white ,densely dotted] 
  (X)--(Y);

 \clip 
 (0,0)--++(-90:4*1)
 --++(0:1*1)  
 --++(0:2*1) 
  --++(0:1*1)  
    --++(90:1*1)  
      --++(0:6*1)      --++(90:1*1)
      --++(0:1) coordinate (X)       --++(0:2) coordinate (Y)         --++(0:1)  --++(90:1)--++(0:4)--++(90:1);

     \foreach \i in {0,1,3,4,5,7,8,9}
         {	\draw[very thick](\i,-2)--++(-90:4);	}
        \foreach \i in          {13,14,15,17,18}
  {	\draw[very thick](\i,0)--++(-90:4);	}
   \foreach \i in {0,1,...,18}
 {	\path(\i+0.5,0)--++(-90:0.5) coordinate (a\i);
 	\path(\i+0.5,0)--++(-90:1.5) coordinate (b\i);
	\path(\i+0.5,0)--++(-90:2.5) coordinate (c\i);
	\path(\i+0.5,0)--++(-90:3.5) coordinate (d\i);	}


 \draw(a13) node {$1$};
 \draw(b13) node {$2$};

 \draw(a14) node {$1$};
  \draw(a17) node {$1$};
    \draw(c0) node {$1$};
      \draw(c3) node {$1$};
        \draw(c4) node {$1$};
    \draw(d0) node {$2$};
        \draw(d3) node {$2$};
        
        \draw(c7) node {$1$};        
        \draw(c8) node {$2$};        
       \draw(c9) node {$2$};                
 
 \end{tikzpicture}$$
 
 \caption{
Generic examples of semistandard tableaux for $\mu_2-m=1 $   
  and $\alpha$   maximal  and near maximal in the dominance ordering.
There are an odd number of 2s (and therefore an even number of 1s) in the pink region.  
 Again, we emphasise  that for a fixed value $m$, the \emph{ only} choices 
 we can make are in the pink region.}
\end{figure}
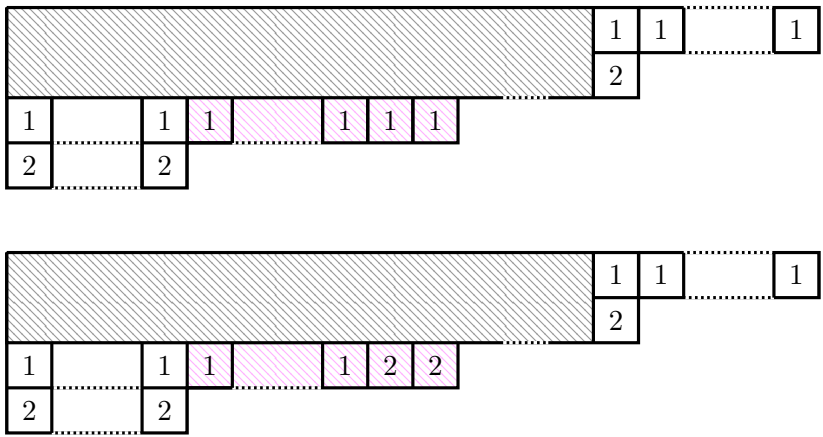

\bigskip
\noindent{\bf Calculating the righthand-side of equation \ref{alabelforaneqn}. }
We first remark that if  $\mu \in \gamma P$ then $\mu$  has 4 non-zero rows, whereas $\gamma$ only has 3 rows.  Therefore $\mu \setminus \gamma$ has an odd fourth row; hence  all its rows are odd and $(\mu -\gamma) \in O_4(n)$.
This implies that $\gamma$ must have only even rows.  
We note that precisely one of the two summands in \ref{formula} has $\gamma_2$ even.  In particular, for $k$ even the righthand sum is zero (since $\gamma_2=k+1$ is odd) and 
 for $k$ odd the lefthand sum is zero (since $\gamma_2=k$ is odd).  
Moreover, 
we are summing over all $S^{(d,d)}_\mathbb{C}\otimes S^{(d+k,d-k)}_\mathbb{C} $ for  $1\leq k \leq d$ and we notice  that the non-zero term in the $k$th even case can be rewritten as follows 
$$    \sum_{0\leq i \leq k}  [S^{ (k+i,k,i)P}_\mathbb{C}] = 
 [S^{(k,k)P}_\mathbb{C}]+ [S^{(2k,k,k)P}_\mathbb{C}] +  \sum_{0<i <k}  [S^{(k+i,k,i)P}_\mathbb{C}] 
   $$
 where the  final sum on the righthand-side is equal to the non-zero term in the $(k-1)$th case, namely  
 $$    \sum_{i=1}^{k-1} [S^{((k-1)+1+i,(k-1)+1,i)P}_\mathbb{C}].
   $$
 Thus, it remains to understand the non-empty even partitions $\gamma=2(c+j,c,j)$    such that $\mu -\gamma  $ is itself a partition; we denote this set by 
   $$D_{\mu}=\{\gamma=2(c+j,c,j)\mid \gamma \neq \varnothing  \text { and } 
  \mu -\gamma \text{ is a partition}\}.$$
 By the above and \ref{formula}, we have that 
\begin{equation}\label{subsitute}
\langle  [S^{(d,d)}_\mathbb{C}\otimes M^{(d+1,d-1)}_\mathbb{C}] \mid [S^\mu_\mathbb{C}] \rangle 
=
2|\{\gamma\in D_{\mu}\mid     \gamma_3\not \in \{0,\gamma_2\}\}|
+
|\{\gamma\in D_{\mu}\mid   \gamma_3 \in\{0, \gamma_2\}\}|
.\end{equation}
We have that $\beta=\mu-\gamma$ is a partition if and only 
\begin{itemize}
\item $\beta_1-\beta_2\geq 0$ which is equivalent to $j\leq \tfrac{1}{2}(\mu_1-\mu_2)$; and  
\item  $\beta_3-\beta_4\geq 0$ which is equivalent to $j\leq \tfrac{1}{2}(\mu_3-\mu_4)$; and  
\item $\beta_2-\beta_3\geq 0$ which is equivalent to $c \leq  \tfrac{1}{2}(\mu_2-\mu_3)+j$.
\end{itemize}
Putting the first two equalities together with the fact that $ (c+j,c,j)$ is itself a partition,  we deduce that 
$0\leq j\leq  \tfrac{1}{2}\min\{\mu_3-\mu_4,\mu_1-\mu_2\}$. 
Putting the final equality together with the fact that $ (c+j,c,j)$ is itself a partition,  we deduce that $j\leq c \leq  \tfrac{1}{2}(\mu_2-\mu_3)+j$.  
Thus we have that 
\begin{align*}
2|\{\gamma\in D_{\mu}\mid  \gamma_3 \not \in \{0, \gamma_2\}				\}|
&=2|\{2(c+j,c,j)\in D_{\mu}\mid 0\neq  j \neq c\}|
\\
&= 
2\times (\tfrac{1}{2}\min\{\mu_3-\mu_4,\mu_1-\mu_2\}\times \tfrac{1}{2}(\mu_2-\mu_3))
\end{align*}
and setting $j=0$ or $j=c$ respectively, we obtain 
$$
 |\{\gamma\in D_{\mu}\mid   \gamma_3 =0\}|
= \tfrac{1}{2}(\mu_2-\mu_3) ,
\qquad
 |\{\gamma\in D_{\mu}\mid   \gamma_3 =\gamma_2\}|=
  \tfrac{1}{2}\min\{  \mu_1-\mu_2, \mu_3-\mu_4\} .
 $$ 
Substituting these three terms  into the righthand-side of \ref{subsitute}, we obtain
 \begin{align*}    \tfrac{1}{2}
 (\min\{\mu_3-\mu_4,\mu_1-\mu_2\} (\mu_2-\mu_3)
 +
 \min\{\mu_3-\mu_4,\mu_1-\mu_2\}+ (\mu_2-\mu_3))
\end{align*}
as required.  
\end{proof}

\bigskip

Thus putting together Propositions \ref{part1} and  \ref{part2} we immediately  obtain the following rather nice
splitting of the square $[k,k]^2$:

\begin{theorem}\label{thmkksplit}
For $k\in \N$ and $n=2k$, we have
\[
\-
S^2([k,k]) =\sum_{\lambda \in E_4(n)} [\lambda] \;,\quad
A^2([k,k]) =\sum_{\lambda \in O_4(n)}[\lambda]\:.
\]
\end{theorem}


The above theorem has an immediate consequence in algebraic combinatorics.  We recall that $C_k=\frac{1}{k+1}{2k \choose k}$ is the $k$-th Catalan number,
and $f(\al)  =[\al](\id )$ is the number of standard Young tableaux of shape~$\al$.

\begin{cor}{\rm
 Given  $\al $ a partition, 
 we define $$s(\al)=\begin{cases}
1 &\text{ if }\alpha\in E_4(n)\\
-1 &\text{ if }\alpha\in O_4(n)\\
0 &\text{ otherwise}. 
\end{cases}$$ 
We have that 
\begin{align}\label{catalan}
C_k = \sum_{\al \in P(2k)} s(\al) f(\al)  .
\end{align}
}
\end{cor}

\begin{proof}
This follows from Theorem~\ref{thmkksplit}  by
evaluating
$S^2([k,k])-A^2([k,k]) = [k,k]^{(2)}$
at the identity, since $C_k=f{(k,k)}$.
\end{proof}

\begin{example}The fourth Catalan number can be calculated as follows
\begin{align*}
C_4&=
14=1+ 20+ 14+ 56+14-35 - 56
\\&=f({\Yvcentermath1\Yboxdim{4pt}\gyoung(;;;;;;;;)})+f({\Yvcentermath1\Yboxdim{4pt}\gyoung(;;;;;;,;;)})
+f({\Yvcentermath1\Yboxdim{4pt}\gyoung(;;;;,;;;;)})
+f({\Yvcentermath1\Yboxdim{4pt}\gyoung(;;;;,;;,;;)})
+f({\Yvcentermath1\Yboxdim{4pt}\gyoung(;;,;;,;;,;;)})
-f({\Yvcentermath1\Yboxdim{4pt}\gyoung(;;;;;,;,;,;)})
-f({\Yvcentermath1\Yboxdim{4pt}\gyoung(;;;,;;;,;,;)}).
\end{align*}
\end{example}

The splitting of the square $[k,k]^2$ immediately implies the splitting
of $[k,k-1]^2$ by restriction; this splitting is explicitly formulated in the
following result.

\begin{theorem}\label{thmk+1ksplit}
For $k\in \N_0$ and $n=2k+1$,
we let  $E_3O_1(n)$ and $E_1O_3(n)$ denote
the sets of partitions $\la\in P_4(n)$
into three even and one odd parts, and one even and three odd parts,
respectively. 
Then we have
\[
\-
S^2([k+1,k]) = \sum_{\la\in E_3O_1(n)} [\lambda]\:, \quad 
A^2([k+1,k]) = \sum_{\la\in E_1O_3(n)} [\lambda] \:.
\]
\end{theorem}

There are some further closely related cases that
have nice decompositions.
The decomposition of the Kronecker square $[k+1,k-1]^2$
(which is not multiplicity-free)
into its irreducible constituents
can be obtained from the general 2-part formulae,
but is also easily derived from the special formulae for
$[k,k]^2$ and $[k,k]  [k+1,k-1]$.

For $\la\in P_4(n)$,
we denote by $d(\la)$ the number of different parts of $\la$;
note that here we are also taking 0 into account, so this may be
different from the number $r_1$ of removable boxes.

\begin{prop}\label{prop:2partalmostmfKronecker}
Let $k\in \N$, $n=2k$.  Let  $E_2O_2'(n)$ (and $E_2O_2'(n)$, respectively)
denote
the set of partitions $\la\in P_4(n)$ with two even and two odd parts,
where the two odd parts (the two even parts, respectively)
are different (again, 0 is considered as an even part).
Then
\[
\-
[k+1,k-1]^2 =\sum_{\lambda \in E_4(n)\cup O_4(n)} (d(\la)-1) [\lambda]
+\sum_{\lambda \in E_2O_2'(n)}[\lambda]
+\sum_{\lambda \in E_2'O_2(n)}[\lambda]\:.
\]
Setting $\delta_\la=1$ if $\la$ has exactly two odd parts, and $\delta_\la=0$ otherwise, we can formulate equivalently:
\[
g((k+1,k-1),(k+1,k-1),\la) = d(\la)-1-\delta_\la .
\]
\end{prop}



The splitting of the square $[k+1,k-1]^2$ into its symmetric and
alternating part is given in the following result.

\begin{theorem}\label{thmk+1k-1}
For $k\in \N$, $n=2k$, we have
\[
\begin{array}{rcl}
S^2([k+1,k-1])
&=&
\dstyle
\sum_{\lambda \in E_4(n)} (d(\la)-1)[\lambda]
+\sum_{\lambda \in E_2O_2'(n)}[\lambda] \\[10pt]
A^2([k+1,k-1])
&=&
\dstyle
\sum_{\lambda \in O_4(n)}(d(\la)-1)[\lambda]
+\sum_{\lambda \in E_2'O_2(n)}[\lambda]\:.
\end{array}
\]
\end{theorem}

\begin{proof}
We want to use the splitting results for
$[k,k]^2$ and $[k+1,k]^2$ in Theorem~\ref{thmkksplit} and Theorem~\ref{thmk+1ksplit}, as well as the product formula
for $[k,k]  [k+1,k]$ in Proposition~\ref{prop:2partalmostmfKronecker}(3).

For the restriction of $S^2([k+1,k])$ to $S_{2k}=S_n$ we obtain
by Theorem~\ref{thmk+1ksplit} on the one hand
\[
S^2([k+1,k])\down_{S_n} =
\sum_{\mu\in E_3O_1(n+1)}\sum_{A\in \Rem(\mu)} [\mu_A].
\]
On the other hand, by \cref{lem:SA-res} we have    that 
\[
S^2([k+1,k])\down_{S_n} =
S^2([k,k]) + S^2([k+1,k-1]) + [k,k]  [k+1,k-1] .
\]
We know the decomposition of all terms appearing here except for that of $S^2([k+1,k-1])$,
thus we can now determine the multiplicities of all constituents of $S^2([k+1,k-1])$.
First we note that only constituents $\la\in E_4(n)\cup E_2O_2(n)$ can occur.
For $\la \in E_4(n)$, we obtain
\[
\gen{S^2([k+1,k-1]),[\la]} =
|\{(\mu,A)\in E_3O_1(n+1)\times \Rem(\mu) \mid \mu_A=\la \}|-1 = d(\la) -1.
\]
Now take $\la \in E_2O_2(n)$; let $d_{\text{odd}}(\la)$ be the number of different odd parts of $\la$. Then we obtain
\[
\gen{S^2([k+1,k-1]),[\la]} =
|\{(\mu,A)\in E_3O_1(n+1)\times \Rem(\mu) \mid \mu_A=\la \}|-1 = d_{\text{odd}}(\la) -1.
\]
Hence
\[
\begin{array}{rcl}
S^2([k+1,k-1])
&=&
\dstyle
\sum_{\lambda \in E_4(n)} (d(\la)-1)[\lambda] +
\sum_{\lambda \in E_2O_2(n)} (d_{\text{odd}}(\la) -1)[\lambda]
\\[10pt]
&=&
\dstyle
\sum_{\lambda \in E_4(n)} (d(\la)-1)[\lambda] +
\sum_{\lambda \in E_2O_2'(n)}[\lambda],
\end{array}
\]
which proves our claim on $S^2([k+1,k-1])$.
As we already know the decomposition of $[k+1,k-1]^2$ into its constituents,
we immediately obtain that also the stated formula for  $A^2([k+1,k-1])$ is correct.
\end{proof}


We end the section with some further results on the splitting of squares
 to arbitrary 2-part partitions.

First we consider the distribution of hooks in the square of
characters to arbitrary 2-part partitions; note that only hook partitions
of length at most~4 can appear.

\begin{prop}\label{prop:2part-hooks}
Let $n\in \N$, $\la\in P(n)$ with $\ell(\la)=2$.
Then the hook constituents appearing in $\cla^2$
distribute as follows into the symmetric and alternating part.
\begin{enumerate}
\item
For $n\ge 3$, $\la=(n-1,1)$, we have
\[
sg(\la,(n))=1=sg(\la,(n-1,1)), \; ag(\la,(n-2,1^2))=1.
\]
\item
For $n=2k$ and $\la=(k,k)$, we have
\[
sg(\la,(n))=1, \; ag(\la,(n-3,1^3))=1.
\]
\item
For $n\ge 4$, $\la\ne (n-1,1),(k,k)$, we have
\[
sg(\la,(n))=1=sg(\la,(n-1,1), \; ag(\la,(n-2,1^2))=1=ag(\la,(n-3,1^3)).
\]
\end{enumerate}
\end{prop}

\begin{proof}
By Theorem~\ref{thm:smalldepthconst} we know that $[n]$
always appears in $S^2(\cla)$ with multiplicity~1, and $[n-1,1]$
appears in $S^2(\cla)$ with multiplicity~1 if $\la\ne (k,k)$
(and is not a constituent of $\cla^2$ otherwise).
\\
(1) has already been stated in Proposition~\ref{prop:smalldepth}.
\\
(2) and (3)
Since only $(n-2,1^2)$ and $(n-3,1^3)$ need to be considered,
both assertions follow from Propositions~\ref{prop:hooks} and~\ref{prop:evenhooks}.
\end{proof}


By the general formulae we know the multiplicities of the 2-part constituents
in the square of a character to a 2-part partition; the next result tells us that none of these 2-part constituents appear in the alternating part.
For the proof we modify the proof of the special case
contained in \cite[Theorem 1.1(a)]{W-2019};
in particular, we also apply the crucial result proved in
\cite[Lemma 2.1]{W-2019}.

\begin{theorem}\label{thm:2part-A2-no2parter}
Let $n\in \N$ and $\la,\mu\in P(n)$ such that $\ell(\la),\ell(\mu) \le 2$.
Then we have
\[
ag(\la,\mu)=0.
\]
\end{theorem}

\begin{proof}
If $\ell(\la)=1$  the result is trivial.
So from now on we will assume that $\ell(\la)=2$.
For $k$ with $0\le 2k\le n$ we set $\chi_k=[n-k,k]$
and $\chi_k^A=A^2([n-k,k])$.
The proof of our assertion is now equivalent to
showing that for all $1\le l \le n/2$ and $0\le k\le n/2$:
\[
\gen{\chi_k,\chi_l^A}=0.
\]

For $0\le j\le n/2$, we let $\pi_j$ be the character to
the permutation character of the natural action of $S_n$ on
$j$-element subsets of  $\{1,\ldots,n\}$.
By \cite{Saxl} we know that $\pi_j$ decomposes as
\[\pi_j = \sum_{k=0}^j \chi_k \:.\]
Now for all $k\le j$,
$\gen{\pi_j,\chi_l^A}\ge \gen{\chi_k,\chi_l^A} \ge 0$,
hence if we can show that $\gen{\pi_j,\chi_l^A}=0$ for some $j$,
then $\gen{\chi_k,\chi_l^A}=0$ for all $k\le j$.

We now set out to prove $\gen{\pi_j,\chi_l^A}=0$ for all $l,j$ such that
$1\le l\le j\le n/2$
(of course, $j=\lfloor n/2 \rfloor$ would be sufficient, but the
argument does not become easier);
this will then imply $\gen{\chi_k,\chi_l^A}=0$ for all $k\le n/2$.
Using $\chi_l=\pi_l-\pi_{l-1}$ and the explicit formula for $\chi_l^A$,
this amounts to showing for $l\le j \le n/2$:
\[
0=\sum_{g\in S_n} \pi_j(g) (\pi_l(g)^2 +\pi_{l-1}(g)^2-2\pi_l(g)\pi_{l-1}(g)-\pi_l(g^2)+\pi_{l-1}(g^2)).
\]
Now, as observed in \cite{W-2019}, the character values of the permutation characters
on $g^2$ can be rewritten using suitable character values on $g$.
Namely, if $\pi_l^{\{2\}}$ is the permutation character to the action of $S_n$
on unordered pairs of $l$-element subsets of
$\{1,\ldots,n\}$, then
\[
\pi_l(g^2)=2 \pi_l^{\{2\}}(g)-\pi_l(g)^2 + 2\pi_l(g) \:.
\]
Using this, our rephrased aim now is to show:
\[
0=\textstyle\sum_{g\in S_n} \pi_j(g) \left(\pi_l(g)^2 -\pi_l(g)\pi_{l-1}(g)-\pi_l^{\{2\}}(g)
+\pi_{l-1}^{\{2\}}(g)-\pi_l(g)+\pi_{l-1}(g)\right).
\]
This is equivalent to seeing that
\[
I_{j,l} := \gen{\pi_j,\pi_l^2 -\pi_l\pi_{l-1}-\pi_l^{\{2\}}
+\pi_{l-1}^{\{2\}}-\pi_l+\pi_{l-1}}=0.
\]
From the decomposition of the $\pi_i$'s, we obtain immediately (using $1\le l\le j$):
\[
I_{j,l} = \gen{\pi_j,\pi_l^2 -\pi_l\pi_{l-1}-\pi_l^{\{2\}}
+\pi_{l-1}^{\{2\}}}-1.
\]
Now we use \cite[Lemma 2.1]{W-2019}; this says that for $1 \le l \le j\le n/2$ we have
\[
\gen{\pi_j,\pi_l^2 -\pi_l\pi_{l-1}-\pi_l^{\{2\}}
+\pi_{l-1}^{\{2\}}} =1 .
\]
Thus we conclude
that  $
I_{j,l} = 0 \; \text{for } 1 \le l \le j\le n/2,
$
and hence our result is proved.
\end{proof}

\begin{rem}{\rm
We note that the analogous question for partitions of width at most~2
(i.e., with largest part at most~2)
is much less interesting as all constituents in  $[\la]^2$, for $\la$ of width at most~2,
have length at most~4. Hence width~2 constituents can only occur for $n\le 8$, and we
can answer the question by simple computations.
Indeed, from the results on the constituents to $(1^n)$ and $(2,1^{n-2})$,
we know when they occur as constituents in $A^2(\cla)$; 
there are no further exceptional cases.
In summary,
$ag(\la,\mu)=0$ for $\la,\mu$ of width at most~2, except for
$(\la,\mu)$ being one of
$((2,1),(1^3))$, $((2,1^2),(2,1^2))$, $((2^2),(1^4))$, $((2^2,1),(2,1^3))$.
Note that the second case is the only one of a width~2 partition $\la$
with $ag(\la,\la)>0$; this is \cite[Theorem 1.1(b)]{W-2019}.
}
\end{rem}

\begin{rem}
It would be very interesting if one could obtain a complete description of the $S^2([\la])$ and $A^2([\la])$  for  $\la\in P_2(n)$.  
\end{rem}

    \section{Irreducible and homogeneous symmetric and antisymmetric products}

The proof of the following theorem utilises many of our earlier results, including   the analysis of small depth and hook cases (which in turn depended on our introduction of 2-modular techniques).

\begin{theorem}
Any symmetric product $S^2([\la])$ for $\la \in P(n)$ is (reducible and) inhomogeneous unless $\la$ is a linear partition.  
Any anti-symmetric product $A^2([\la])$ for $\la \in P(n)$ is (reducible and) inhomogeneous unless $\la=(n), (n-1,1)$, $(2^2)$, or  $(3^2)$  (up to conjugation).  

We have that $A^2([n-1,1]) = [n-2,1^2]$ is irreducible for all $n\ge 3$; furthermore 
  $A^2([2^2])=[1^4]$ and 
$A^2([3^2])=[3,1^3]$.
\end{theorem}

\begin{proof}
That the  listed products  have the stated form can be easily checked by hand for the $n=4$ and $6$ cases and 
$A^2([n-1,1]) = [n-2,1^2]$ follows from    \cref{prop:smalldepth}.    
We now turn to the main part of the theorem, verifying that these are the only homogeneous products.  We begin with the symmetric case.  
Assume that  $\la$ is a non-linear partition.  
  For all $\la \in P(n)$, the trivial partition always labels   a constituent of $S^2([\la])$ with multiplicity equal to 1.
    Thus we need only note that 
 the  degree of the character of the  symmetric square is $\tfrac{1}{2}(\chi^2(e)+\chi(e))>1$ and so there must be some other, non-isomorphic, constituent  as required.

We now turn to the harder case of the anti-symmetric Kronecker products.  
If $\la  \in P(n)$ and $n\leq 9$ we can check the result by hand, and so we now assume 
   that $n>9$ and $\la$ is not of the form $\la=(n-1,1)$, $(2^2)$,  $(3^2)$, or  a linear partition.
    If $\la$ is a hook partition   then  $A^2([\la])$  contains all constituents of the form $(n-m,1^m)$ for $m\equiv 2,3 $ modulo $4$ (by Theorem \ref{thm:MW}) and so the result follows.   
If $  \la\in P_2(n)$ and $\la \neq (n/2,n/2)$, then the result follows from Proposition \ref{prop:2part-hooks}(iii) and if $\la=(n/2,n/2)$ then the result follows from Theorem \ref{thmkksplit}.

We now consider the case that  $\la=(a^b)$ is a rectangle for some $a,b>2$.  We already know that $\langle A^2[(a^b)] \mid [ab-3,1^3]\rangle =1$ and so it suffices to show that 
$A^2[(a^b)] \neq [ab-3,1^3]$.  To do this, we need only find a conjugacy class on which the characters do not coincide.  
We set $\alpha=(a+b,1^{ab-a-b})$ if $a+b$ is odd and $\alpha=(a+b+1,1^{ab-a-b-1})$ if $a+b$ is even.  
Let $g\in S_{ab}$ be of cycle type $\alpha$.  
Since all parts of $\alpha$ are odd, we have that $\chi(g)=\chi(g^2)$ for any character $\chi$ of $S_n$.  
All hooks of $(a^b)$ are strictly smaller than $a+b$ and so,   by the Murnaghan--Nakayama rule,  
 $ [(a^b)] (g) =0=  	 [(a^b)] (g^2)	$ and this implies that $A^2[(a^b)](g)=0$.  
 On the other hand, $3<a+b, a+b+1<ab-3$ and so 	 $[ab-3,1^3](g)=[ab-3-|\alpha|,1^3](1^{ab-|\alpha|})\neq 0$, again by the Murnaghan--Nakayama rule.  Thus the characters do not agree, as required.

We now assume that $\la$ is not of any one of the above forms (so $\la$ has at least 2 removable nodes and is not a hook or 2-line partition).   
We set $r_1=|{\rm Rem}(\la)|\geq 2$, by assumption.  
We have that    
$A^2(\la)$ contains $[n-2,1^2]$ with multiplicity ${r_1 \choose 2}\geq1$ by \cref{thm:smalldepthconst}.  
We will show that 
 $\langle A^2([\la]) \mid \chi_{\rm hook}\rangle > {r_1 \choose 2}$ 
and hence deduce that $A^2([\la])$ has another, distinct, hook constituent;  the result will then follow.  
Since $\la$ is not itself a hook, we can put together 
\Cref{prop:hooks,thm:smalldepthconst} and hence deduce that 
\begin{align*} 
\langle A^2([\la]) \mid \chi_{\rm hook}\rangle=
&		\langle S^2([\la]) \mid \chi_{\rm hook}\rangle
= 1+ (r_1-1)+  {\textstyle{r_1-1 }\choose 2}  + \dots 
\end{align*}
 where $\dots$ denotes the contribution from hook constituents of depth greater than 2.  Now, we have that 
 $$
  1+ (r_1-1)+  {\textstyle{r_1-1 }\choose 2} =
  \frac{2r_1+(r_1^2-3r_1+2)}{2}
  =
    \frac{ r_1^2-r_1+2}{2}
    >
        \frac{ r_1^2-r_1 }{2}= {r_1 \choose 2}
  $$as required.  
\end{proof}

\begin{rems}{\rm
In \cite{MM}, Malle and Magaard consider
the symmetric and alternating parts of the tensor squares
of irreducible modules for the alternating groups in all characteristics,
and they classify when these are irreducible.
}\end{rems}


  \section{Multiplicity-free symmetric and alternating parts}\label{sec:mf-parts}

Multiplicity-free Kronecker products have been
classified in \cite{BeBo}; fortunately, the
classification of the multiplicity-free Kronecker squares
is much easier, and we recall this here.

\begin{prop}\cite[Proposition 4.1]{BeBo}  \label{prop:squares}
Let $\lambda$ be a partition of~$n$.
Then $[\la]^2$ is multiplicity-free
if and only if $\la$ is one of the following (up to conjugation):
$$(n), \; (n-1,1), \;  (\left\lceil\frac n2 \right\rceil , \left\lfloor  \frac n2 \right\rfloor)\:.$$
  \end{prop}

Obviously, in the cases arising above, both $S^2(\cla )$ and $A^2(\cla )$ are multiplicity-free,
and we have seen in all cases how the Kronecker square
decomposes into the symmetric and alternating part. 
In fact we find:

\begin{prop} \label{prop:mfsymalt}
Let $n\in \N$.
\begin{enumerate}
\item[{(1)}]
If $\la$ (or its conjugate) is one of the partitions
\[
(n), \; (n-1,1), \;  (\left\lceil\frac n2 \right\rceil , \left\lfloor  \frac n2 \right\rfloor)
\]
then $S^2(\cla)$ is multiplicity-free.

\item[{(2)}]
If $\la\in P(n)$ is such that $S^2(\cla)$ is multiplicity-free then $\la$ (or its conjugate) is a rectangle or one of the partitions
\[
(n-1,1), \text{or } (k+1,k) \: \text{with } n=2k+1.
\]

\item[{(3)}]
If $\la$ (or its conjugate) is one of the partitions
\[
\begin{array}{l}
\dstyle
(n), \; (n-1,1), \;  (n-2,2), \; (n-2,1^2), \;
(\left\lceil\frac n2 \right\rceil , \left\lfloor  \frac n2 \right\rfloor),
\\[5pt]
\text{or one of the exceptional partitions } (5,3), (3^3),
\end{array}
\]
then $A^2(\cla)$ is multiplicity-free (or zero).

\item[{(4)}]
If $\la\in P(n)$ is such that $A^2(\cla)$ is multiplicity-free then
\color{black} $\la$ has at most 2 removable nodes.
\end{enumerate}
\end{prop}

\begin{proof}
Parts (1) and (3): By the comments above and the results in this section,
we have already seen that for $(n), \; (n-1,1),
(\left\lceil\frac n2 \right\rceil , \left\lfloor  \frac n2 \right\rfloor)$
(and their conjugates) the Kronecker square is multiplicity-free,
and for the partitions $(n-2,2)$ and $(n-2,1^2)$ (and their conjugates)
the alternating part of the square is multiplicity-free.
The case $\la=(5,3)$ follows from Theorem~\ref{thm:depth3},
and the remaining case $\la=(3^3)$ was checked by direct
computation. 
%
%
%
%
%

Parts (2) and (4) are an immediate consequence of Theorem~\ref{thm:smalldepthconst}.
\end{proof}

 We are now able to classify the situations in which both
the symmetric and   alternating parts
are multiplicity-free:

\begin{theorem} \label{prop:mfsym-and-alt}
Let $n\in \N$, $\la\in P(n)$.
Then both $S^2(\cla)$ and $A^2(\cla)$ are multiplicity-free if and only if
$\la$ (or its conjugate) is one of the partitions
\[
(n), \; (n-1,1), \;  (\left\lceil\frac n2 \right\rceil , \left\lfloor  \frac n2 \right\rfloor).
\]
\end{theorem}

\begin{proof}
We already know that for the partitions $\la$ listed in the assertion the
symmetric and alternating part of $\cla^2$ are multiplicity-free.

Conversely, we now assume that both parts  of $\cla^2$ are multiplicity-free,
but that $\la$
is not of one of the listed forms.
 By Proposition~\ref{prop:mfsymalt},
we only have to consider rectangular partitions $\la$
of diagonal length at least~3.
 
 Since $\la$ is not a hook, the proof of \cite[Proposition 4.1]{BeBo}
 immediately shows that $\cla^2$ has a constituent of multiplicity at least~3, unless possibly $(3^3)\subseteq \la$.  Thus (at least) one of $S^2(\cla)$ and $A^2(\cla)$  has a constituent with multiplicity greater than or equal to 2.  
 
For $\la = (3^3)$, the product $[\la]^2$ does not contain any constituent with multiplicity equal to 3, however we can directly compute $\gen{S^2([3^3]),[5,2^2]} = \gen{([3^3])^2,[5,2^2]} =2$.

We can now assume that $\la \supset (3^3)$ and $\ell(\la)\ge 4$.
 Set $\tilde\la = \la \cap (3^4)$; then $\tilde\la=(3^3,j)$, $j\in \{1,2,3\}$, and in all these cases $[\tilde\la]^2$ has a constituent of multiplicity~3,
 hence by the monotonicity of Kronecker coefficients,
 $\cla^2$ has a constituent of multiplicity at least~3.
\end{proof}


Computational data 
lead to the following stronger classification
conjecture.

\begin{conj}\label{prop:mfsym-and-alt2}
The partitions listed in Proposition~\ref{prop:mfsymalt}(1) and (3)
are all the partitions
where the symmetric or the alternating part of the square is multiplicity-free.
\end{conj}

\begin{rems}
{\rm
(1)
If the conjecture holds, then a multiplicity-free symmetric part $S^2(\cla)$
implies that $\cla^2$ is multiplicity-free, while a
multiplicity-free alternating part $A^2(\cla)$
implies that the coefficients in $\cla^2$ are at most~2.

(2) From Theorem~\ref{thm:MW} we already know that no further
hooks can appear in the classification.

%
}\end{rems}


  \section{Splitting the square: Refining the Saxl Conjecture}\label{sec:Saxl-refined}


As we want to discuss refinements of the Saxl conjecture,
we first formally recall the conjecture:

\begin{conj}\label{conj:Saxl}(Saxl's conjecture.)
For any $k\in \N$, the Kronecker square  $[\rho_k]^2$ contains \emph{ all } $[\la] \in \Irr(S_n)$ as constituents.
\end{conj}

As mentioned in the introduction,
there is a conjecture (not restricted to triangular numbers)
due to Heide, Saxl, Tiep and Zalesski \cite{HSTZ},
which says that
for any $n\ne 2,4,9$ there is some character $\psi\in\Irr(S_n)$
such that $\psi^2$ contains all $\chi \in \Irr(S_n)$ as constituents.
The Saxl conjecture suggests a special candidate for triangular numbers $n$.
Apart from computational results for (relatively) small $k$,
there are already many contributions towards this conjecture that
confirm that several families of constituents occur in the square
(see e.g.\ \cite{B-spinSaxl, BB, I-2015, PP, PPV}).


We have seen above that we may even refine the Saxl conjecture
for certain constituents, as we have found the location (or multiplicities)
of some constituents
in the symmetric or alternating part, respectively.
For example, apart from the constituents at the extreme ends, we had found
in Corollary~\ref{cor:rho} that $[\rho_k]$ is
a constituent of the symmetric part of~$[\rho_k]^2$.

Of course, we cannot expect that $A^2([\rho_k])$
contains all irreducible $\chi\in \Irr(S_n)$
as it does not contain the trivial character.
But computational data suggest that while
both characters $S^2([\rho_k])$ and $A^2([\rho_k])$
do not contain all irreducible
characters as constituents,
the number of missing irreducible characters is (perhaps surprisingly)
small.

By Theorem~\ref{thm:GIP}, $S^2([\rho_k])$ does not have $[1^n]$ as a constituent
when $k\equiv 2 \mod 4$.
Also, we know that $A^2([\rho_k])$ never contains $[n],[n-1,1]$ and $[n-2,2]$,
and it does not contain $[1^n]$ when $k\not\equiv 2 \mod 4$.

Based on computational data we suggest the following
strengthening of Saxl's conjecture:

\begin{conj} {\bf (Refinement of the Saxl Conjecture)}
The symmetric part  $S^2([\rho_k])$ of the square $[\rho_k]^2$
contains all irreducible characters $[\la]$ of $S_n$ as constituents,
except for the character $[1^n]$ when $k\equiv 2 \mod 4$.

The alternating part  $A^2([\rho_k])$ of the square $[\rho_k]^2$
contains almost all irreducible characters $[\la]$ of $S_n$ as constituents,
 with the only missing characters being $[n]$, $[n-1,1]$ and $[n-2,2]$, and  furthermore $[1^n]$ when $k\not\equiv 2 \mod 4$,
 and in addition $[2^3]$ when $k=3$.
\end{conj}

As stated earlier, Saxl's conjecture is a refinement in the case of triangular numbers $n$ of a more general conjecture due to Heide, Saxl, Tiep and Zalesski \cite{HSTZ}, which says that for any $n\ne 2,4,9$  
there is some $\la\in P(n)$ such that $\cla^2$ contains all irreducible characters of $S_n$ as constituents.

In fact, the computational data suggest that also this ``HSTZ''-conjecture has a strengthening; remember that $[1^n]$ is a constituent of $\cla^2$ if and
only if $\la$ is symmetric, and then we know its position in $S^2(\cla)$ or
$A^2(\cla)$ from Theorem~\ref{thm:GIP}.



\begin{conj}{\bf (Refinement of the HSTZ-Conjecture)}
For any $n\ge 10$ there is some $\la \in P(n)$
such that the symmetric part  $S^2(\cla)$ of the square $\cla^2$
contains all irreducible characters of $S_n$ as constituents,
except possibly for the character $[1^n]$.


For any $n\ge 10$ there is some $\mu \in P(n)$ such that the
 alternating part  $A^2(\cmu)$ of the square $\cmu^2$
contains almost all irreducible characters of $S_n$ as constituents,
 with the only missing characters being $[n]$, $[n-1,1]$ and $[n-2,2]$, and  furthermore possibly $[1^n]$.


Furthermore,
for any $n\ge 10$  there is always a symmetric partition $\la$ that is optimal
for both $S^2(\la)$ and  $A^2(\la)$ (in the sense above).

Also, for a symmetric partition $\la$ that is optimal for $A^2(\cla)$,
i.e., missing out only the 3 or 4 stated constituents,
the square $\cla^2$ contains all irreducible characters.
\end{conj}

We note explicitly that not every symmetric partition $\la$
with the property that $\cla^2$ contains all irreducible
characters as its constituents is optimal for both $S^2(\cla)$
and $A^2(\cla)$. For example, $(4^2,2^2)$ is
(in the sense above) not optimal for both parts of its square,
but $(5,3,2,1^2)$ and $(6,2,1^4)$ both are ``doubly optimal''.

    \bigskip 
\noindent{\bf Acknowledgments.}
 This research was conducted over  several enjoyable summers 
      spent by the second author in  Hannover, made possible by funding from the 
 Alexander von Humboldt Foundation    and EPSRC early career grant EP/V00090X/1.   
The second author is  grateful to 
 Harm Derksen for making available his code for computing (anti-)symmetric Kronecker coefficients and to Christian Ikenmeyer for 
 teaching him how to use it.  
 We also thank Mike Zabrocki for providing the outline of an alternative  proof of  Theorem B, due to 
  John Stembridge.


\end{document}